\documentclass[10pt,a4paper,notitlepage,headinclude,twoside]{article}
\usepackage[utf8]{inputenc}
\usepackage[british]{babel}
\usepackage{xargs}                      
\usepackage[pdftex,dvipsnames]{xcolor}  

\usepackage{blindtext}   
\usepackage{geometry}
\usepackage{amsmath, amsthm, amsfonts, amssymb}
\usepackage[alphabetic]{amsrefs}   
\usepackage{authblk}   
\usepackage{MnSymbol}  
\usepackage{mdframed}
\usepackage{bm}        
\usepackage{slashed}     
\usepackage{xcolor}    
\usepackage[pdftex]{graphicx}
\usepackage[all]{xy}  
\usepackage{ifthen}
\usepackage{tikz}
\usepackage{sseq}
\usepackage{index}
\usepackage{epigraph} 
\usepackage{slashed}     
\usepackage[colorinlistoftodos,prependcaption,textsize=tiny]{todonotes} 
\usepackage[english, refpage]{nomencl}                    
\usetikzlibrary{matrix,arrows}

\usepackage{nameref}
\makeatletter
\newcommand*{\currentname}{\@currentlabelname}
\makeatother
\usepackage{fancyhdr}  
\usepackage{wrapfig}   

\usepackage{hyperref}  
\hypersetup{colorlinks=true, linkcolor=blue, citecolor=blue, urlcolor=blue,}

\theoremstyle{plain} 
\newtheorem{proposition}{Proposition}[section]
\newtheorem{theorem}[proposition]{Theorem}
\newtheorem{lemma}[proposition]{Lemma}

\newtheorem{thmx}{Theorem} 

\theoremstyle{definition}
\newtheorem{definition}[proposition]{Definition}

\theoremstyle{remark} 
\newtheorem{rem}[proposition]{Remark}

\newcommandx{\unsure}[2][1=]{\todo[linecolor=red,backgroundcolor=red!25,bordercolor=red,#1]{#2}}
\newcommandx{\change}[2][1=]{\todo[linecolor=blue,backgroundcolor=blue!25,bordercolor=blue,#1]{#2}}
\newcommandx{\info}[2][1=]{\todo[linecolor=OliveGreen,backgroundcolor=OliveGreen!25,bordercolor=OliveGreen,#1]{#2}}
\newcommandx{\improvement}[2][1=]{\todo[linecolor=Plum,backgroundcolor=Plum!25,bordercolor=Plum,#1]{#2}}
\newcommandx{\thiswillnotshow}[2][1=]{\todo[disable,#1]{#2}}

\newcommand{\pr}{\mathord{\mathrm{pr}}}                 
\newcommand{\diff}{\mathrm{d}\,}                     

\newcommand{\Z}{\mathord{\mathbb{Z}}}                   
\newcommand{\Q}{\mathord{\mathbb{Q}}}                   
\newcommand{\R}{\mathord{\mathbb{R}}}                   
\newcommand{\C}{\mathord{\mathbb{C}}}                   
\newcommand{\I}{\mathord{\mathrm{i}}}                   
\newcommand{\J}{\mathord{\mathrm{j}}}                   
\newcommand{\Quat}{\mathord{\mathbb{H}}}                
\newcommandx{\Cliff}[1][1=]{\mathrm{Cl}_{#1}}

\newcommand{\CP}{\mathord{\mathbb{C}P}}

\newcommand{\id}{\mathord{\mathrm{id}}}

\newcommand{\Diff}{\mathrm{Diff}}
\newcommand{\hAut}{\mathrm{hAut}}

\newcommandx{\Or}[1][1=d, usedefault]{\mathrm{O}(#1)}
\newcommandx{\SOr}[1][1=d, usedefault]{\mathrm{SO}(#1)}
\newcommandx{\Un}[1][1=d, usedefault]{\mathrm{U}(#1)}
\newcommandx{\SUn}[1][1=d, usedefault]{\mathrm{SU}(#1)}
\newcommandx{\PUn}[1][1=d, usedefault]{\mathrm{PU}(#1)}
\newcommandx{\Spin}[1][1=d, usedefault]{\mathrm{Spin}(#1)}
\newcommandx{\GL}[1][1=d, usedefault]{\mathrm{GL}(#1)}
\newcommandx{\Sp}[1][1=d, usedefault]{\mathrm{Sp}(#1)}

\newcommand{\vertical}{\mathrm{vert}}

\newcommandx{\gtorpedo}[1][1= , usedefault]{g_{\mathrm{torp #1}}}
\newcommand{\ground}{g_{\circ}}
\newcommand{\gFubiniStudy}{g_{FS}}
\newcommandx{\gDO}[1][1= , usedefault]{{g_{\mathcal{DO}(1) #1}}}
\newcommand{\gDOext}{{g_{\mathcal{DO}(1)}^{\mathrm{ext}}}}
\newcommand{\scal}{\mathrm{scal}}
\newcommand{\sect}{\mathrm{sec}}
\newcommandx{\Riem}[1][1={}, usedefault]{\mathrm{Riem}^{#1}}
\newcommandx{\Moduli}[2][1=M , 2={\mathrm{scal}>0}, usedefault ]{\mathrm{Riem}^{#2}(#1)/ \mathrm{Diff}(#1)}
\newcommandx{\HModuli}[2][1=M, 2={\mathrm{scal}>0}, usedefault ]{\mathrm{Riem}^{#2}(#1)/\!\!/ \mathrm{Diff}(#1)}
\newcommandx{\ObsModuli}[3][1=M, 2={\mathrm{scal}>0} ,3={x_0}, usedefault]{\mathrm{Riem}^{#2}(#1)\!/ {\mathrm{Diff}_{x_0}(#1)}}
\newcommandx{\HObsModuli}[3][1=M, 2={\mathrm{scal}>0} ,3={x_0}, usedefault]{\mathrm{Riem}^{#2}(#1)/\!\!/ {\mathrm{Diff}_{x_0}(#1)}}

\newcommand{\placeholder}{\,\text{-}\,}

\usepackage[T1]{fontenc}

\makenomenclature
\makeindex

\title{Moduli Spaces of Positive Curvature Metrics in Dimension Four and Beyond}
\author{Thorsten Hertl}
\affil[1]{Department of Mathematics and Statistics, University of Melbourne, Parkville, VIC 3010, Australia}

\date{\small{\texttt{thorsten.hertl@unimelb.edu.au}}}

\begin{document}
\maketitle
\begin{abstract}
    We construct non-trivial elements in the homotopy groups of the observer moduli space of positive sectional curvature metrics on $\mathbb{C}P^n$ and non-trivial elements in the homotopy groups of the observer moduli space of positive scalar curvature metrics on $\mathbb{C}P^2 \sharp M^4$.
\end{abstract}

\vspace{3pt}
\hspace{7.5pt} \textbf{MSC classes:} 58D27 (primary), 53C21, 55Q52, 55R15, 57R22 (secondary). 


\section{Introduction}

During the last decades, the space $\Riem[C](M)$ of Riemannian metrics satisfying a curvature condition $C$, like positive sectional curvature or positive scalar curvature, has attracted a lot of attention, see for example \cite{hitchin1974harmonic}, \cite{carr1988construction},  \cite{HSS2014PSCSpace}, \cite{botvinnik2017infinite}, and \cite{TuschmannWraith2015ModuiSpaces} for a more comprehensive survey.
That these spaces may be non contractible were observed for the first time by Hitchin \cite{hitchin1974harmonic}, who used information about homotopy groups of diffeomorphism groups and its pullback action on the space of positive scalar curvature metrics to construct `exotic positive scalar curvature metrics' on $8k$-dimensional spheres and `exotic loops' of those.

One could argue that it is more natural to study the moduli space $\Moduli[M][C]$ instead because it is this space that actually carries significant geometric information without the shadow of the topology of the underlying manifold tripling through.
Whether or not one agrees with this opinion, it is unarguably true that the pullback action $\Diff(M) \curvearrowright \Riem[C](M)$ is not free in general, so the quotient space is often not well behaved.

Therefore, the focus has shifted towards the \emph{observer moduli space} $\ObsModuli[M][C]$ obtained by dividing out the smaller group $\Diff_{x_0}(M)$ of \emph{observer diffeomorphisms} consisting of all diffeomorphisms that fix an initially given point $x_0$ and whose differential at $x_0$ is the identity.
The pullback action $\Diff_{x_0}(M) \curvearrowright \Riem[C](M)$ is indeed free and the quotient $\ObsModuli[M][C]$ is closely related to the classifying space $B\Diff_{x_0}(M)$.
This observation was used in \cite{botvinnik2010homotopy} to construct the first examples of non-trivial elements in higher homotopy groups of observer moduli spaces of positive scalar curvature metrics.
These arguments were subsequently refined by \cite{Botvinnik2019ObserverModRicci} and \cite{reiser2023generalization} to observer moduli spaces of positive Ricci curvature metrics and certain positive intermediate curvature metrics, respectively.

The main result of this article provides to the authors knowledge the first examples of non-trivial elements in higher homotopy groups of observer moduli spaces of positive sectional curvature metrics.

 \begin{thmx}\label{Main Thm - PSec}
     Let $n \geq 2$ and let $3 \leq j \leq n-1$ be odd. 
     Let further $\gFubiniStudy$ be the Fubini-Study metric on $\CP^n$. 
     Then we have the following results:
     \begin{itemize}
         \item[(1)] $\pi_2(\ObsModuli[\CP^n][\sect > 0],[\gFubiniStudy])$ contains an element of order at least $n+1$ if $n$ is even and an element of order at least $(n+1)/2$ if $n$ is odd.
         \item[(2)] $\pi_{2j}(\ObsModuli[\CP^n][\sect > 0],[\gFubiniStudy]) \otimes \Q \neq 0$.
     \end{itemize} 
     Moreover, if $C$ is an open, diffeomorphism invariant curvature condition that is implied by positive sectional curvature, then the same results hold for $\mathrm{Riem}^{C}(\CP^n)/\mathrm{Diff}_{x_0}(\CP^n)$.
 \end{thmx}

 We wish to emphasise that all previously obtained examples require the degree of the homotopy groups to be small compared to the dimension of the manifold while our examples do not.
 Furthermore, we do not require these manifolds to be spin, so they are detected without the help of index theory.

 Although the examples of Theorem \ref{Main Thm - PSec} are established using homotopy theory, the employed construction also has a geometric interpretation.
 For positive scalar curvature metrics, which experience more flexible gluing constructions, this geometric interpretation allows us to prove that observer moduli spaces of positive scalar curvature metrics have non-trivial second homotopy groups for a large class of oriented, connected four manifolds.

 \begin{thmx}\label{Main Thm - PSC}
     If $M^4$ is a closed, oriented, connected, smooth manifold of dimension four that carries at least one positive scalar curvature metric $g_0$, then
     \begin{equation*}
         \pi_2\left(\ObsModuli[\CP^2\sharp M],[\gFubiniStudy \sharp g_0]\right) \neq 0.
     \end{equation*}
 \end{thmx}
 
 If we restrict our attention to simply connected manifolds and take their signature into account, then we can strengthen the previous results.

 \begin{thmx}\label{Main Thm - PSC with order}
     If $M^4$ is a closed, oriented, simply-connected, smooth manifold of dimension four that carries at least one positive scalar curvature metric $g_0$ and has negative signature $\mathrm{sign}(M) =: -\sigma <0$, then
     \begin{equation*}
         \pi_2\left(\ObsModuli[\CP^{2\, \sharp \sigma} \sharp M],[\gFubiniStudy^{\sharp \sigma} \sharp g_0]\right) \otimes \mathbb{Q} \supseteq \mathbb{Q}^{\sigma}.
     \end{equation*}
     For manifolds with positive signature, the statement remains true if  we replace $\CP^2$ by $\overline{\CP}^2$, the complex projective plane with the opposite orientation.
 \end{thmx}
 
 The paper is structured as follows: We first provide in Section \ref{Section - Foundations} foundational material of (homotopy) moduli spaces of positively curved metrics that differs slightly from the presentations in the existing literature. 
 Next, we construct in Section \ref{Section - Bundle Construction} a non-trivial smooth fibre bundle $\mathcal{DO}(1)$ over $S^2$ with fibre $D\mathcal{O}(1)$, the disc bundle of the dual tautological line bundle over $\CP^1$, that is trivial near the boundary. 
 We use this bundle in Section \ref{Section - Positive Scalar Curvature } to establish the non-trivial elements in the observer moduli spaces mentioned in Theorem \ref{Main Thm - PSC} and \ref{Main Thm - PSC with order}.
 In the final Section \ref{Section - Positive Sectional Curvature}, we present a more geometrically flavoured proof of the first part of Theorem \ref{Main Thm - PSec} and a more homotopy-theoretical one of the second part.\\

 \noindent \textbf{Acknowledgements:} The author would like to thank Sebastian Goette for helpful and inspiring conversations that led to the bundle construction in Section \ref{Section - Bundle Construction}, and the anonymous referee for reading the article with such a great care and for providing many improving suggestions.



\section{Foundations on Moduli Spaces}\label{Section - Foundations}

 We shortly recall the theory and different versions of moduli spaces and construct the fibre bundles we will need in subsequent sections.
 Unless stated otherwise, $M$ denotes a smooth, closed, connected, and oriented manifold of dimension $d$.

 Let $\Diff(M)$ be the topological group of all orientation preserving diffeomorphisms and let $\Riem(M)$ be the space of Riemannian metrics both equipped with their usual $C^\infty$-topology.
 The diffeomorphism group $\Diff(M)$ acts from the right on $\Riem(M)$ and all subspaces $\Riem[C](M)$ via pullback, where $C$ denotes an open, diffeomorphism invariant curvature condition. 
 However, we will use the induced left action by pulling back along their inverses instead.
 The only relevant curvature conditions for us in this article are positive sectional curvature $C = \sect > 0$ and positive scalar curvature $C = \scal > 0$, which yield open, diffeomorphism-invariant subspaces of $\Riem(M)$.
 The quotient $\Moduli[M][C]$ is called the \emph{moduli space} of positive sectional/scalar curvature metrics.

 Usually, the action of $\Diff(M)$ on $\Riem[C](M)$ is not free.
 One possibility to solve this issue is to use the subgroup of \emph{observer diffeomorphisms} 
 \begin{equation*}
     \Diff_{x_0}(M) := \{\varphi \in \Diff(M) \,:\, \varphi(x_0) = x_0 \text{ and } D_{x_0}\varphi = \id\},
 \end{equation*}
 because this subgroup acts freely.
 Indeed, since geodesics are uniquely determined by their initial values, isometric observer diffeomorphisms are the identity on a geodesic neighbourhood of $x_0$. 
 An open-and-closed argument now implies that isometric observer diffeomorphisms are the identity on the entire connected component of $x_0$.
 The quotient $\ObsModuli[M][C]$ is called the \emph{observer moduli space}.

 Another approach is to enlarge $\Riem[C](M)$ so that the induced action becomes free:
 Assume that a topological group $G$ acts on a topological space $X$ (not necessarily freely).
 Let $G \rightarrow EG \rightarrow BG$ be the universal $G$-principal bundle.
 The space $EG$ is contractible, so $EG \times X$ has the same homotopy type as $X$, but the diagonal action is now free.
 The Borel construction $X/\!\!/ G := EG \times_G X$ is a fibre bundle over $BG$ with fibre $X$ and a continuous, equivariant map $f \colon X \rightarrow Y$ induces a map of fibre bundles $F = \id \times_G f \colon X/\!\!/ G \rightarrow Y/\!\!/ G$ over the identity of $BG$.
 
 A bit less known is the following statement, which is not hard to prove.
 \begin{lemma}\label{lem - hquot vs BG}
   Under the identification $\pi_n(BG) \cong \pi_n(\Omega B G) \cong \pi_{n-1}(G)$ via fibre-transportation, the boundary operators in the induced long exact sequence of homotopy groups agree with the homomorphism induced by the orbit map. 
   In other words, we have the following commutative diagram
   \begin{equation*}
       \xymatrix{\pi_{n-1}(G,1) \ar@{=}[r] \ar@/_1pc/[rr]_-{[g] \mapsto [g_\bullet x_0]} & \pi_n(BG,\ast) \ar[r]^-{\partial_n}  & \pi_{n-1}(X,x_0). }
   \end{equation*}
 \end{lemma}
 \begin{proof}
     Denote the action of $G$ on $X$ by $\rho$.
     Since $EG \cong EG \times_G G$, the map $\rho$ gives rise to a commutative diagram of fibre bundles:

     \begin{equation*}
         \xymatrix@R-0.5em@C+3em{ G \ar[r] \ar[d]^{\rho(\placeholder,x_0)} & EG\times_G G \ar[r] \ar[d]^{\id \times_G \rho(\placeholder,x_0)} & BG \ar@{=}[d] \\
                    X \ar[r] & X/\!\!/G \ar[r] & BG.}
     \end{equation*}
     The claim now follows from the fact that the fibre-transport induces the boundary operators in the long exact sequence of homotopy groups and that these induced long exact sequences are natural with respect to maps of fibre bundles.
 \end{proof}
 Finally, we would like to compare the homotopy quotient to the actual quotient.
 To this end, note that the the projection onto the second component yields a comparison map $X/\!\!/G \rightarrow X/G$.
 \begin{lemma}\label{lem - HQuot vs Quot}
    If the action on $X$ is free, proper, and the quotient map $X \rightarrow X/G$ has a local section around each point, then the canonical comparison $X/\!\!/G \rightarrow X/G$ is a weak-homotopy equivalence.
 \end{lemma}
 \begin{proof}
     Denote the action $G\curvearrowright X$ by $\rho$.
     By a local application of \cite{tomDieck2008algebraic}*{Proposition 14.1.5}, we deduce that the canonical projection $X \rightarrow X/G$ turns $X$ into a $G$-principal bundle.
     This and Lemma \ref{lem - hquot vs BG} implies that the following commutative diagram
     \begin{equation*}
         \xymatrix{ G \ar[rr]^{\rho(\placeholder,x_0)} \ar@{=}[d] && X \ar[rr]^{x \mapsto [e_0,x]} \ar@{=}[d] && EG \times_G X \ar@{=}[r] \ar[d] & X/\!\!/G \\
         G \ar[rr]^{\rho(\placeholder,x_0)} && X \ar[rr] && X/G}
     \end{equation*}
     induces maps between two long exact sequences of homotopy groups (associated to the horizontal sequences). 
     The Five-Lemma implies now that $\pi_k(X/\!\!/G,[e_0,x_0]) \rightarrow \pi_k(X/G,[x_0])$ is an isomorphism for all $k\geq 1$.

     It remains to show that $X/\!\!/G \rightarrow X/G$ induces a bijection on path-components.
     To see this, observe that the existence of a local section to $X \rightarrow X/G$ around each point implies that this projection has the path lifting property.
     The path lifting property in turn implies that the canonical map $\pi_0(X)/\pi_0(G) \rightarrow \pi_0(X/G)$, which is always surjective, is also injective.
     In addition, the projection to the second component $\pr_2 \colon EG \times X \rightarrow X$ is an equivariant map and a weak homotopy equivalence, because $EG$ is contractible; in particular, it induces a bijective $\pi_0(G)$-equivariant map $\pi_0(\mathrm{pr}_2) \colon \pi_0(EG \times X) \rightarrow \pi_0(X)$, which in turn yields a bijection $\pi_0(\mathrm{pr}_2) \colon \pi_0(EG \times X)/\pi_0(G) \rightarrow \pi_0(X)/\pi_0(G)$.
     
     The fact that $X/\!\!/G \rightarrow X/G$ induces a bijection on path components now follows from the following commutative diagram
     \begin{equation*}
         \xymatrix{ \pi_0(EG \times X)/\pi_0(G) \ar[d]^{\mathrm{pr}_2}_-{\cong} \ar[r]^-\cong & \pi_0(X/\!\!/G) \ar[d] \\ 
         \pi_0(X)/\pi_0(G)  \ar[r]^\cong & \pi_0(X/G).}
     \end{equation*}\qedhere
 \end{proof}
 If $X = \Riem[C](M)$ and $G = \Diff(M)$ or $\Diff_{x_0}(M)$ for some fixed $x_0 \in M$, we call $\HModuli[M][C]$ the \emph{homotopy moduli space} and $\HObsModuli[M][C]$ the \emph{homotopy observer moduli space}, respectively.
 We are going to show how to obtain (homotopy classes of) continuous maps $X \rightarrow \HObsModuli[M][C]$ from $M$-fibre bundles $E \rightarrow X$ that carry a Riemannian metric $g$ with curvature condition $C$ on their vertical tangent bundle.
 We first show that the homotopy moduli space $\HModuli[M][C]$ is a classifying space of a homotopy-functor, see Lemma \ref{lem - Classifying Property of HModuli} and then establish the connection between homotopy observer moduli spa $\HObsModuli[M][C]$ and our object of interest $\HModuli[M][C]$ in Proposition \ref{Prop - HModuliProperties}, which will show, in particular, that $\HObsModuli[M][C]$ is a also a classifying space for a homotopy functor.
 In order to relate the homotopy observer moduli space $\HObsModuli[M][C]$ to $\ObsModuli[M][C]$, we need to know that the canonical projection $\Riem[C](M) \rightarrow \ObsModuli[M][C]$ is a $\Diff_{x_0}(M)$-principal bundle.
 We establish this fact in Lemma \ref{Lemma - Observer Moduli Space Fibre bundle}.

 An intriguing model for $B\Diff(M)$ and hence for $\HModuli[M][C]$ that allows to interpret it as a manifold version of the infinite Grassmannian manifold $B\SOr[d] = \mathrm{Gr}^+_{d}(\R^\infty)$ is presented in \cite{botvinnik2017infinite}:
 Whitney's embedding theorem implies that the space of embeddings $\mathrm{Emb}(M,\R^\infty)$ is a (weakly) contractible topological space, on which $\Diff(M)$ acts freely and properly from the right.
 One can show that the quotient map $\mathrm{Emb}(M,\R^\infty) \rightarrow \mathrm{Emb}(M,\R^\infty)/\Diff(M)$ has a local section, so $\mathrm{Emb}(M,\R^\infty)$ is a model for $E\mathrm{Diff}(M)$ and
 \begin{equation*}
     B\Diff(M) = \mathrm{Emb}(M,\R^\infty)/\Diff(M) = \{N \subseteq  \R^\infty \, : \, N \cong M\}
 \end{equation*}
 can be modelled by the space of all submanifolds of $\R^\infty$ that are diffeomorphic to $M$.
 The associated universal $M$-fibre bundle in this picture is
 \begin{equation*}
     M_{\mathrm{univ}} := \mathrm{Emb}(M;\R^\infty) \times_{\Diff(M)} M \cong \{(N,p) \,:\, p \in N \subseteq \R^\infty, N \cong M \}.
 \end{equation*}
 Note that we can recover $\mathrm{Emb}(M,\mathbb{R}^\infty)$ as the $\Diff(M)$-principal bundle $E \rightarrow B\Diff(M) = \{ N \subseteq \mathbb{R}^\infty\, : \, N \cong M\}$ whose fibre over $N$ is given by the set of all diffeomorphisms $M \rightarrow N$.
 Thus, similar to the tautological bundles over infinite Grassmann manifolds, this bundle induces a bijection 
 \begin{equation*}
     [X,B\Diff(M)] \xrightarrow{1:1} \mathrm{Bndl}_M(X), \qquad [f] \mapsto f^\ast M_{\mathrm{univ}}.
 \end{equation*}
 from the set of homotopy classes $[X,B\Diff(M)]$ into the set of isomorphism classes of numerable $M$-fibre bundles over $X$.
 
 Using this model of $B\Diff(M)$, the homotopy moduli space can now be described by
 \begin{equation*}
     \HModuli[M][C] = \{(N,g) \, : \, M \cong N \subseteq \R^\infty, \, g \in \Riem[C](N)\}.
 \end{equation*}
 Note that the Riemannian metric $g$ is not required to arise from the embedding of $N$.
 There is also a universal $M$-fibre bundle over $\HModuli[M][C]$ that is given by the homotopy quotient
 \begin{equation*}
     M_{\mathrm{univ}}^C := \mathrm{Emb}(M;\R^\infty) \times_{\Diff(M)} (\mathrm{Riem}^C(M) \times M),
 \end{equation*}
 where the diffeomorphism group acts on the product diagonally from the right, on $\mathrm{Riem}^C(M)$ with the pullback action, and on $M$ with the inverse of the tautological action.
 Note that this bundle agrees with the pullback of $M_{\mathrm{univ}}$ along the projection $p \colon \HModuli[M][C] \rightarrow B\Diff(M)$, because the canonical map from the homotopy quotient into the pullback is a continuous map over the identity of $\HModuli[M][C]$ and fibrewise a diffeomorphism.
 The bundle $M_{\mathrm{univ}}^C$ has a canonical Riemannian metric $h$ satisfying curvature condition $C$ on its vertical tangent bundle $T^\vertical M_{\mathrm{univ}}^C = E\Diff(M) \times_{\Diff(M)} (\mathrm{Riem}^C(M) \times TM)$ that is given by the formula
 \begin{equation*}
     h^{\mathrm{univ}}_{[\iota,g,m]}([\iota,g,v],[\iota,g,w]) = g_{m}(v,w).
 \end{equation*}
 This metric is well defined because 
 \begin{align*}
     & \quad \, h_{[\iota \circ \varphi,\varphi^\ast g, \varphi^{-1}(m)]}\left([\iota \circ \varphi, \varphi^\ast g, D_m\varphi^{-1} (v)],[\iota \circ \varphi, \varphi^\ast g, D_m\varphi^{-1} (w)]\right) \\
     &= \bigl(\varphi^\ast g\bigr)_{\varphi^{-1}(m)}(D_m\varphi^{-1} v, D_m \varphi^{-1} w) = g_m\bigl( D_{\varphi^{-1}(m)} \varphi D_m \varphi^{-1} v, D_{\varphi^{-1}(m)} \varphi D_m \varphi^{-1} w \bigr)\\
     &= g_m(v,w) = h_{[\iota,g,m]}([\iota,g,v],[\iota,g,w]).
 \end{align*}

 Of course, $\HObsModuli[M][C]$ has a similar description.
 In particular, over this space exists the universal $M$-fibre bundle with structure group $\Diff_{x_0}(M)$ equipped with a fibrewise Riemannian metric satisfying condition $C$.

 Like $B\Diff(M)$, the homotopy moduli space, $\HModuli[M][C]$, is a classifying space for a homotopy functor: 
 Let $f \colon X \rightarrow \HModuli[M][C]$ be a continuous map.
 As the space $\HModuli[M][C]$ has a universal Riemannian metric satisfying curvature $C$ on the vertical tangent bundle $T^\vertical M^C_{\mathrm{univ}} = p^\ast T^\vertical M_{\mathrm{univ}}$ it can be pulled back to a 
 fibrewise Riemannian metric on $f^\ast M^C_{\mathrm{univ}}$, which we denote by $h^\mathrm{univ}\circ f$ inspired from the explicit construction of a pullback.
 \begin{lemma}\label{lem - Classifying Property of HModuli}
   Let $X$ be a topological space. 
   The map that assigns to a homotopy class $[f] \in [X;\HModuli[M][C]]$ the homotopy class of sections $[h^\mathrm{univ}\circ f] \in \mathrm{Riem}_{\vertical}(f^\ast M_{\mathrm{univ}}^C)$ yields a one to one correspondence
   \begin{equation*}
       \xymatrix{ [X;\HModuli[M][C]] \ar[r]^-{1:1} & \underset{[E]}{\coprod} \pi_0\left( \Riem[C]_{\mathrm{vert}}(E) \right)\bigr/\sim_{\mathrm{Iso}}, }
   \end{equation*}
   where $[E]$ runs over the set of isomorphism classes of numerable $M$-fibre bundles over $X$, where $\Riem[C]_{\mathrm{vert}}(E)$ denotes the space of Riemannian metrics on the vertical tangent bundle  $T^\vertical E$, and where $\sim_{\mathrm{Iso}}$ denotes the equivalence relation that identifies two classes $[g_0], [g_1] \in \pi_0\left( \Riem[C]_{\mathrm{vert}}(E) \right) $ if and only if there is a bundle automorphism of $E$ over the identity of $X$ that pulls back a representative of $[g_1]$ to a representative of $[g_0]$.
 \end{lemma}
 \begin{proof}


    We start with the proof that the map is well-defined: 
    Assume that $f_0$ is homotopic to $f_1$ via $H \colon X \times [0,1] \rightarrow \HModuli[M][C]$, say, then $h^{\mathrm{univ}} \circ H \in \mathrm{Riem}^C_{\vertical}(H^\ast M^C_{\mathrm{univ}})$.
    By the homotopy theorem for fibre bundles, see \cite{tomDieck2008algebraic}*{Theorem 14.3.1}, there is a bundle isomorphism\footnote{Note that the fibre automorphism $\Phi(\placeholder,1)$ might not be the identity on $f_1^\ast M^C_{\mathrm{univ}}$ even if $f_0 = f_1$. This phenomenon always occurs if $H$ is not homotopic to the constant homotopy relative to $X \times \{0,1\}$.}  $\Phi \colon  f^\ast_0 M^C_{\mathrm{univ}} \times [0,1] \rightarrow H^\ast M^C_{\mathrm{univ}}$ over the identity of $X \times [0,1]$, such that $\Phi(\placeholder,0)$ is the identity.
    The map $\Phi(\placeholder,1) \colon f_0^\ast M^C_{\mathrm{univ}} \rightarrow f_1^\ast M^C_{\mathrm{univ}}$ is a bundle isomorphism and the pullback $\Phi(\placeholder,1)^\ast (h^{\mathrm{univ}}\circ f_1)$ is some fibre metric there, which is homotopic to $h^{\mathrm{univ}}\circ f_0$ via $\Phi(\placeholder,t)^\ast( h^{\mathrm{univ}}\circ H(\placeholder,t) )$ as $\Phi(\placeholder,0)=\id$. 

    Regarding surjectivity, let $E \rightarrow X$ be a numerable $M$-fibre bundle over $X$.
    As $B\mathrm{Diff}(M)$ is the classifying space of numerable $M$-fibre bundles, there is a map $f \colon X \rightarrow B\Diff(M)$ such that $f^\ast M_{\mathrm{univ}} \cong E$.
    We can ignore $E$ if it does not carry a fibre metric satisfying curvature condition $C$ because $\pi_0( \mathrm{Riem}^C(T^\vertical E) )$ is empty then.
    A fibre metric $g$ on $f^\ast M_{\mathrm{univ}} = f^\ast \mathrm{Emb}(M,\R^\infty) \times_{\Diff(M)} M$ is by definition a section into $S^2 T^{\vertical,\vee} f^\ast M_{\mathrm{univ}} =  f^\ast \mathrm{Emb}(M,\R^\infty) \times_{\Diff(M)} S^2 T^\vee M$, and every section is of the form
    \begin{align*}
         \begin{split}
             f^\ast \mathrm{Emb}(M,\R^\infty) \times_{\Diff(M)} M \ni &[(x,\iota),m] \mapsto \\
             &[(x,\iota),g_{(x,\iota),m}] \in f^\ast \mathrm{Emb}(M,\R^\infty) \times_{\Diff(M)} S^2 T^\vee M,
         \end{split}
    \end{align*}
    with $x \in X$ and $\iota \in \mathrm{Emb}(M,\R^\infty)$ an embedding satisfying $f(x) = [\iota]$.
    This map is well defined if and only if $g_{(x,\iota\circ \varphi),\varphi^{-1}(m)}(\placeholder,\placeholder) = g_{(x,\iota),m}(D_{\varphi^{-1}(m)}(\placeholder),D_{\varphi^{-1}(m)}(\placeholder))$, so the map $(x,\iota) \mapsto (m \mapsto g_{(x,\iota),m})$ constitutes to a $\Diff(M)$-equivariant map $f^\ast \mathrm{Emb}(M,\R^\infty) \rightarrow \mathrm{Riem}^C(M) \subseteq \Gamma(S^2 T^\vee M)$, which we denote by $\mathrm{Ad}(g)$.
    Conversely, every equivariant map of this kind yields a fibre metric on $f^\ast M_{\mathrm{univ}}$.
    The map
    \begin{equation*}
       F_g \colon X \rightarrow \HModuli[M][C] \quad \text{given by} \quad x \mapsto \Bigl[ \iota(f(x)) , \mathrm{Ad}\bigl(g)(x,\iota(f(x))\bigr) \Bigr], 
    \end{equation*}
    where $\iota(f(x)) \in \mathrm{Emb}(M,\R^\infty)$ is any lift of $f(x) \in B\Diff(M) = \mathrm{Emb}(M,\R^\infty)/\Diff(M)$, is well defined and continuous.
    Essentially by definition, we have $h^\mathrm{univ}\circ F_g = g$ on $F^\ast_g M^C_\mathrm{univ} = f^\ast M_{\mathrm{univ}}$.

    Injectivity is a relative version of the surjectivity.
    Assume that $[h^\mathrm{univ} \circ f_0] = [h^\mathrm{univ} \circ f_1]$.
    By definition, this means that we can find a bundle isomorphism $\Phi \colon (pf_1)^\ast M_{\mathrm{univ}} \rightarrow (pf_0)^\ast M_{\mathrm{univ}}$ such that the push forward $\Phi_\ast (h^{\mathrm{univ}} \circ f_1)$ lies in the same path component of $\Gamma(f_0^\ast M^C_\mathrm{univ}) = \Gamma((pf_0)^\ast M_\mathrm{univ})$ as $h^\mathrm{univ} \circ f_0$.

    To this given $\Phi$, we find a homotopy $H \colon X \times [0,1] \rightarrow B\Diff(M)$ between $pf_0$ and $pf_1$, and a bundle isomorphism $H^\ast M_{\mathrm{univ}} \cong (pf_0)^\ast M_{\mathrm{univ}} \times [0,1]$ that restricts to the identity on $H^\ast M_{\mathrm{univ}}|_{X \times \{0\}}$ and to $\Phi$ on $H^\ast M_{\mathrm{univ}}|_{X \times \{1\}}$ as follows:
    As all bundle maps $(pf_0)^\ast M_{\mathrm{univ}} \rightarrow M_{\mathrm{univ}}$ are bundle homotopic, see \cite{tomDieck2008algebraic}*{Proposition 14.4.4} for the equivalent statement for $\Diff(M)$-principal bundles, we can find a homotopy of bundle maps between the canonical map $(pf_0)^\ast M_{\mathrm{univ}} \rightarrow M_{\mathrm{univ}}$ and the composition of $\Phi^{-1} \colon (pf_0)^\ast M_{\mathrm{univ}} \rightarrow (pf_1)^\ast M_{\mathrm{univ}}$ and $(pf_1)^\ast M_{\mathrm{univ}} \rightarrow M_{\mathrm{univ}}$, which is denoted by $\mathcal{H}$.
    The induced map on the base space $H = \mathrm{bs}(\mathcal{H}) \colon X \times [0,1] \rightarrow B\Diff(M)$ is a homotopy between $pf_0$ and $pf_1$.
    By the universal property of the pullback, the map $\mathcal{H}$ induces a bundle isomorphism over the identity 
    \begin{equation*}
        \xymatrix{ (pf_0)^\ast M_{\mathrm{univ}} \times [0,1] \ar[rr]^-{\hat{\mathcal{H}}}_-\cong \ar@/_1pc/[rrr]_{\mathcal{H}} && H^\ast M_{\mathrm{univ}}  \ar[r] & M_{\mathrm{univ}}, }
    \end{equation*}
    and a close inspection shows that $\hat{\mathcal{H}}$ restricts to the identity on $(pf_0)^\ast M \times \{0\}$ and to $\Phi^{-1}$ on $(pf_0)^\ast M \times \{1\}$, so its inverse is the trivialisation we are looking for.

    We find, by assumption, a path of fibre metrics satisfying curvature condition $C$ between $h^\mathrm{univ} \circ f_0$ and $\Phi_\ast(h^\mathrm{univ} \circ f_1) = (\Phi^{-1})^\ast (h^\mathrm{univ} \circ f_1)$, and we use this path to equip $(pf_0)^\ast M_{\mathrm{univ}} \times [0,1]$ with a fibre metric satisfying $C$.
    The pushforward of this metric with $\hat{\mathcal{H}}$ is a fibre metric on $H^\ast M_{\mathrm{univ}}$ that restricts to $h^\mathrm{univ} \circ f_0$ on $X \times \{0\}$ and to $h^\mathrm{univ} \circ f_1$ on $X \times \{1\}$.
    We now proceed as in the proof of surjectivity to construct a homotopy between $f_0$ and $f_1$.
 \end{proof}

 Before we compare the homotopy observer moduli space to the homotopy moduli space and the observer moduli space, it is reasonable to work out the difference between $B\Diff(M)$ and $B\Diff_{x_0}(M)$.
 We consider the local case first, and, to this end, consider the group $\Diff(D^d,S^{d-1})$ of all diffeomorphisms on the disc $D^d$ that restrict to the identity near $S^{d-1}$.
 Note that we can interpret $\Diff(D^d,S^{d-1})$ as a subgroup of $\Diff(M)$ using local coordinates centred at $x_0$.

 \begin{lemma}\label{Lemma - Differential Extension}
    Every orientation preserving, linear automorphism $A \in \mathrm{GL}_d^+(\R)$ extends to a diffeomorphism $F \in \Diff(D^d,S^{d-1})$ such that $D_0F = A$.
\end{lemma}
\begin{proof}
    We first restrict to the special case that $A = \mathrm{diag}(\lambda_1,\dots,\lambda_n)$ is a positive definite, diagonal matrix.
    Set $\Delta := \mathrm{diag}\left(\log(\lambda_1),\dots,\log(\lambda_n)\right)$ and choose a smooth function $\chi \colon [0,1] \rightarrow [0,1]$ that is constantly $1$ on $[0,1/10]$ and constantly $0$ on $[1/2,1]$.
    The smooth vector field $X^\Delta$ on $D^d$ given by $X^\Delta(x) = \chi(||x||) \cdot \Delta \cdot x$ is supported in $1/2 D^d$ and vanishes at the origin.
    It follows that its flow $\Phi^{X^\Delta}_t$, which constitutes to a continuous map $D^d \times \R_{\geq 0} \rightarrow D^d$, fixes the origin, and restricts to the identity on the complement of $1/2D^d$.

    Choose a $\varepsilon>0$ such that $\Phi^{X^\Delta}(\varepsilon D^d \times [0,1]) \subseteq 1/10 D^d$, which exists, because $\Phi^{X^\Delta}$ is continuous and $\Phi^{X^\Delta}(\{0\} \times [0,1]) = \{0\}$.
    Since the vector field $X^\Delta$ agrees with the vector field $x \mapsto A\cdot x$ on $1/10 D^d$, uniqueness of the solution to the flow-equation implies that, on $\varepsilon D^d$, the flow satisfies 
    \begin{equation*}
        \Phi^{X^\Delta}_t(x) = \exp(t\Delta)\cdot x = \mathrm{diag}\bigl( e^{t\log(\lambda_1)},\dots, e^{t\log(\lambda_n)}\bigr) \cdot x.
    \end{equation*}

    Thus, the map $F_A := \Phi^{X^\Delta}_1$ is a diffeomorphism that restricts to the identity outside of $1/2D^d$, fixes the origin, and satisfies $D_0F_A = \mathrm{diag}(\lambda_1,\dots,\lambda_n) = A$ as desired.

    \vspace{6pt}

    The next general case is that $A$ is positive definite and self-adjoint. 
    We find an orthogonal matrix $U$ such that $U^\ast A U = \mathrm{diag}(\lambda_1,\dots,\lambda_n)$.
    It is easy to check that $F_A$ defined by $F_A(x) = UF_{U^\ast A U}(U^\ast x)$ is a diffeomorphism that is the identity on the complement of $1/2D^d$, fixes the origin, and satisfies $D_0F_A = A$.

    For the most general case $A \in \mathrm{GL}^+_d(\R)$, we use the polar decomposition theorem to find an orthogonal matrix $U$ such that $A = U|A|$ with $|A| = \sqrt{A^\ast A}$ a positive definite and self-adjoint matrix.
    Since $A$ is orientation preserving, $U \in \SOr[d]$, and we can find a smooth path $H \colon [0,1] \rightarrow \SOr[d]$ with $H(t) = U$ on $[0,1/2]$ and $H(t) = \id_{\R^d}$ near $1$.
    The map $F_A$ defined by $F_A(x) = H(||x||) \cdot F_{|A|}(x)$ is clearly smooth. 
    It is a diffeomorphism because it restricts to the diffeomorphism $U\cdot F_{|A|} \colon 1/2D^d \rightarrow 1/2D^d$ and on $D^d\setminus 1/2D^d$ to the (isometric) diffeomorphism\footnote{Isometric in the sense that $||H(||x||)x|| = ||x||$, it is of course not an isometry in the sense of Riemannian geometry.} $x \mapsto H(||x||) \cdot x$.
    By construction of $F_{|A|}$, we have $F_A(0) = 0$, $D_0F_A = U D_0F_{|A|} = U|A| = A$. As $H(t) = \id_{\R^d}$ near $\{1\}$, the map $F_A$ is the identity in a neighbourhood of $S^{d-1}$.
\end{proof}

 Note that we cannot make the choices in the proof of Lemma \ref{Lemma - Differential Extension} continuously depending on $A$, because this would imply that $\mathrm{GL}_d^+(\R)$ is contractible.
 However, we can always find a local section.
 
 \begin{lemma}\label{Lemma - Local Section for Differential}
    The map $\Diff(D^d,S^{d-1}) \rightarrow \R^d \times \mathrm{GL}_d^+(\R)$ given by $\varphi \mapsto (\varphi(0),D_0\varphi)$ has a local section $\sigma$ around $(0,\id)$.
 \end{lemma}
 \begin{proof}
    To construct a local section around $(0,\id) \in D^d \times \mathrm{GL}_d^+(\R)$, we choose a function $\chi \colon [0,1] \rightarrow [0,1]$ such that $\chi = 1$ on $[0,1/100]$ and $\chi = 0$ on $[2/100,1]$, and $|\chi'| < 200$.
    We choose a sufficiently small ball $B_\varepsilon(0)$ around $0 \in \R^{d \times d}$ such that the map $A \mapsto \mathrm{exp}(A)$ is invertible and that, for all $A \in B_\varepsilon(0)$, we have  $||\mathrm{exp}(A) - \id||\leq \mathrm{exp}({||A||}) - 1 < 2\varepsilon$.
    Set $V_\varepsilon = \mathrm{exp}(B_\varepsilon(0))$.
    As the exponential function $\mathrm{exp} \colon \R^{d \times d} \rightarrow \mathrm{GL}^+_d(\R)$ is a local diffeomorphism, every element of $V_\varepsilon$ can be written as $\exp(A)$ in a unique fashion.

   The local section $\sigma \colon 1/10^{5} \cdot D^d \times V_\varepsilon \rightarrow \Diff(D^d,S^{d-1})$ is now given by
   \begin{equation*}
       (y,\mathrm{exp}(A)) \mapsto \bigl( D^d \ni x \mapsto \mathrm{exp}(\chi(||x||)A)\cdot x + \chi(||x||) \cdot y \bigr).
   \end{equation*}
   
   By definition of $\chi$, we have that $\sigma(y,\mathrm{exp}(A))(0) = y$ and $D_0\sigma(y,\exp(A)) = \exp(A)$, as well as that the map $\sigma(y,\mathrm{exp}(A))$ is the identity on $\{2/10 < ||x|| \leq 1\}$.
   A straightforward calculation, in which we abbreviate $\chi(||x_j||)$ to $\chi(x_j)$, shows that
   \begin{align*}
       & \quad \ \sigma(y,\mathrm{exp}(A))(x_1) - \sigma(y,\mathrm{exp}(A))(x_2)  \\
       &=  \exp(\chi(x_1)A)x_1 - \exp(\chi(x_2)A) )x_2 + (\chi(x_1) - \chi(x_2))y \\
       \begin{split}&= (x_1 - x_2) + \sum_{n=1}^\infty \frac{A^n}{n!} \Bigl( \chi(x_1)^n(x_1 - x_2) + (\chi(x_1) - \chi(x_2)) \sum_{j=0}^{n-1} \chi(x_1)^j\chi(x_2)^{n-j-1} \cdot x_2 \Bigr) \\
       &\quad \ + (\chi(x_1) - \chi(x_2))y.
       \end{split}
   \end{align*}
   Since $0 \leq \chi(||x||) \leq 1$, and $|\chi(||x_1||) - \chi(||x_2||)| \leq 200 ||x_1 - x_2||$, we deduce, under the assumption that $||A|| < \varepsilon$ is sufficiently small, from the reversed triangle inequality that 
   \begin{align*}
       & \quad \ ||\sigma(y,\mathrm{exp}(A))(x_1) - \sigma(y,\mathrm{exp}(A))(x_2)|| \\
       &\geq ||x_1 - x_2|| - \Bigl( 200\cdot ||y|| + \sum_{n=1}^\infty \frac{||A||^n}{n!}\bigl( 1 + 200 n \cdot ||x_2|| \bigr) \Bigr) \cdot ||x_1 - x_2|| \\
       &\geq 1/2 \cdot ||x_1 - x_2|| > 0,
   \end{align*}
   and hence injectivity.
   
   The differential of $\sigma(y,\mathrm{exp}(A))$ is given by\footnote{Since $\chi' = 0$ near zero, the expression also makes sense at $x=0$.} 
   \begin{align*}
      & \quad \, D_x\sigma(y,\mathrm{exp}(A))(h) \\
      &= \mathrm{exp}(\chi(x)A)\cdot h + \Bigl(\chi'(||x||)A\mathrm{exp}(\chi(||x||)A)x + \chi'(||x||)y\Bigr)\langle x/||x||,h\rangle
   \end{align*}
   which is an isomorphism if $\varepsilon$ is sufficiently small.
   Hence $\sigma(y,\mathrm{exp}(A))$ is an injective, open map with a compact, connected manifold with boundary as domain.
   Hence, $\sigma(y,\mathrm{exp}(A))$ is a diffeomorphism. 
 \end{proof}

 We denote by $\mathrm{Fr}^+(M)$ the $\mathrm{GL}_d^+(\R)$-principal bundle of positive oriented frames on $M$.

 \begin{lemma}\label{Lemma - Frame bundle as diffeo quotient}
    Let $f_0 \in \mathrm{Fr}^+(M)_{x_0}$ be a positively oriented basis of $T_{x_0}M$. 
    The map $\Diff(M) \rightarrow \mathrm{Fr}^+(M)$ given by $\varphi \mapsto D_{x_0}\varphi \cdot f_0$ yields a homeomorphism $\Diff(M)/\Diff_{x_0}(M) \cong \mathrm{Fr}^+(M)$.
    In addition, the map $\Diff(M) \rightarrow \mathrm{Fr}^+(M)$ is a $\Diff_{x_0}(M)$-principal bundle.
\end{lemma}
\begin{proof}
   Since the normal subgroup $\Diff(M)_0$ of all diffeomorphisms that are isotopic to the identity acts transitively on a connected manifold, and since, by Lemma \ref{Lemma - Differential Extension}, every orientation preserving linear isomorphism $A \colon T_{x_0}M \rightarrow T_{x_0}M$ extends to an orientation preserving diffeomorphism $F$ with fixpoint $x_0$ and differential $D_{x_0}F=A$, the group of orientation preserving diffeomorphisms $\Diff(M)$ acts transitively on $\mathrm{Fr}^+(M)$. 
   Its isotropy group at $f_0 \in \mathrm{Fr}^+(M)$ is precisely $\Diff_{x_0}(M)$, which implies $\Diff(M)/\Diff_{x_0}(M) \cong \mathrm{Fr}^+(M)$.

   To prove that $\Diff(M) \rightarrow \Diff(M)/\Diff_{x_0}(M) \cong \mathrm{Fr}^+(M)$ is a $\Diff_{x_0}(M)$-principal bundle, it is enough to provide a local section around the identity coset (which corresponds to the base point $f_0 \in \mathrm{Fr}^+(M)_{x_0}$).
   Using local charts centred around $x_0$, the problem reduces to the special case $M^d = D^d$ and $x_0=0$, so that $\mathrm{Fr}^+(M) = D^d \times \mathrm{GL}_d^+(\R)$.
   Thus, the local section is provided by Lemma \ref{Lemma - Local Section for Differential}.   
\end{proof}

 We now turn our attention to the observer moduli space.
 As we have observed above, the pullback action of $\Diff_{x_0}(M)$ on $\mathrm{Riem}^C(M)$ is free, so we expect that $\mathrm{Riem}^C(M) \rightarrow \ObsModuli[M][C]$ is a fibre bundle.
 This expectation turns out to be true, and we are going to prove it using a version of Ebin's slice theorem for the observer diffeomorphism group, which we are going to derive from Ebin's classical slice theorem.

 Recall that, for a continuous right action $\rho \colon G \curvearrowright X$, a \emph{slice} $S_x$ around $x \in X$ is a subspace containing $x$ and satisfying the following three conditions:
 \begin{itemize}
     \item[SL1] $S_x$ is $\mathrm{Stab}_G(x)$-invariant.
     \item[SL2] If $S_x \cap S_x\cdot g \neq \emptyset$, then $g \in \mathrm{Stab}_G(x)$.
     \item[SL3] There is an open subset $U$ around the unit coset $[1] \in \mathrm{Stab}_G(x)\backslash G$ and a section $\chi \colon U \rightarrow G$ such that the following map induced by the right action 
     \begin{equation*}
         \mu \colon S_x \times U \rightarrow X, \qquad (s,[g]) \mapsto \rho\bigl(s,\chi([g])\bigr) = s \cdot \chi([g]).
     \end{equation*}
     is an embedding with open image.
 \end{itemize}
 \begin{rem}\label{rem:Independence-of-sections}
     Note that SL3 is no special condition on the section $\chi$ in the sense that if $\chi_2 \colon U \rightarrow G$ is another section, then $\mu_2(s,u) := s \cdot \chi_2(u)$ also defines an embedding with open image.
     Indeed, as $\chi$ and $\chi_2$ are two sections of the projection $G \rightarrow \mathrm{Stab}_G(x) \backslash G$, there is a (unique) continuous function $\gamma \colon U \rightarrow \mathrm{Stab}_G(x)$ with $\gamma(u) \chi(u) = \chi_2(u)$ and $\gamma([1]) = 1$.
     As $S_x$ is $\mathrm{Stab}_G(x)$-invariant by SL1, the map 
     \begin{equation*}
        \mu_2(s,u) = s_2 \cdot (\gamma(u) \cdot \chi(u)) =  (s_2 \cdot \gamma(u))\cdot \chi(u) = \mu\bigl( s_2\cdot \gamma(u), \chi(u) \bigr) 
     \end{equation*}
     decomposes into homeomorphism followed by an embedding with open image
     \begin{equation*}
         S_x \times U \xrightarrow{ (s,u) \mapsto (s\cdot \gamma(u),u)} S_x \times U  \xrightarrow{ \mu(s,\chi(u)) } X. 
     \end{equation*}
     Thus, $\mu_2$ is an embedding with open image, too.
 \end{rem}
 \begin{lemma}\label{lem - Slices for good subgroups}
    Let $\rho \colon G \curvearrowright X$ be a continuous right action for which there exists a slice around $x \in X$. 
    Let further $H \leq G$ be a closed subgroup for which the canonical projection $\kappa \colon G \rightarrow G/H$ has a local section around the unit coset.  
    
    If the map $G/H \rightarrow \mathrm{Stab}_G(x)\backslash G/H$ is a $\mathrm{Stab}_G(x)$-principal bundle\footnote{The right action is induced by the canonical left action $\gamma_\bullet [g] := [\gamma^{-1}g]$.}, then the restricted right action $\rho \colon H \curvearrowright X$ also has a slice around $x$.
 \end{lemma}
 \begin{proof}
     Let $S_x$ be a slice of $G$ around $x$ and abbreviate $\mathrm{Stab}_G(x)$ to $\Gamma$.
     Choose local sections from open neighbourhoods of the unit cosets
     \begin{equation*}
         \Gamma\backslash G \supseteq U \xrightarrow{\chi} G, \quad G/H \supseteq V \xrightarrow{\sigma} G, \quad \Gamma\backslash G/H \supseteq W \xrightarrow{\tau} G/H, 
     \end{equation*}
     with $\chi$ a section as described in SL3, such that all of them send unit cosets to unit cosets.
     We further require that $\tau(W) \subseteq V$. 
     If we denote the canonical projection $G \rightarrow \Gamma \backslash G$ with $\pi$, then $\theta := \pi \circ \sigma \circ \tau$ is a local section of $\Gamma \backslash G \rightarrow \Gamma \backslash G/H$.
     For later, we will also assume that $\pi(\sigma(V)) \subseteq U$.

     Since $G/H \rightarrow \Gamma\backslash G/H$ is a $\Gamma$-principal bundle, $\Gamma$ acts freely on $G/H$ from the left, and hence $H$ acts freely from the right on $\Gamma \backslash G$.
     It follows that $\Gamma \backslash G \rightarrow \Gamma \backslash G /H$ is an $H$-principal bundle.

     A priori, the section $\chi \colon U \rightarrow G$ is not equivariant with respect to the right action, so we are going to construct a section $\chi'$ with that property:
     Without loss of generality, we may assume that $\Gamma \backslash G|_W \cong W \times H$ (otherwise we replace $W$ by a sufficiently small open subset of it) via the map $(w,h) \mapsto \theta(w) \cdot h$ and that $\theta(W) \subseteq U$.
     Now, on $\theta(W) \cdot H \cap U$ we define the section $\chi'$ via $\chi'(\theta(w)\cdot h) = \chi(\theta(w)) \cdot h$.

     Finally, define $\Sigma_x := \rho\left( S_x \times \chi(\theta(W))\right)$ and note that the right action $\rho$ yields a homeomorphism between $S_x \times \chi(\theta(W))$ and $\Sigma_x$ by condition SL3.
     We claim that $\Sigma_x$ is indeed a slice.

     As $H$ acts freely on $\Gamma\backslash G \cong x \cdot G$, we deduce that $\mathrm{Stab}_H(x) = \{1\}$, so $\Sigma_x$ is clearly $\mathrm{Stab}_H(x)$-invariant, and condition SL1 is verified.

     Regarding SL2, assume that there is an $h \in H$ such that $\Sigma_x \cap \Sigma_x \cdot h \neq \emptyset$. 
     This implies, by definition, that there are $s_1,s_2 \in S_x$ and $w_1,w_2 \in W$ such that $s_1\cdot\chi(\theta(w_1)) = s_2 \cdot \chi(\theta(w_2))\cdot h$, which in turn implies, by condition SL2 for the slice $S_x$, that 
     \begin{equation*}
         \chi( \theta(w_2)) \cdot h \cdot \chi( \theta(w_1) )^{-1} =: \gamma \in \Gamma \quad \text{ or equivalently } \quad \gamma \cdot \chi(\theta(w_1)) = \chi(\theta(w_2)) \cdot h. 
     \end{equation*} 
     It follows, by applying the canonical projection $G \rightarrow \Gamma \backslash G$ to this identity, that $\theta(w_1) = \theta(w_2) \cdot h \in \Gamma \backslash G$.
     Since $\theta$ is a local section of the projection $\Gamma \backslash G \rightarrow \Gamma \backslash G / H$, we conclude 
     \begin{equation*}
        \theta(w_2)\cdot h= \theta(w_1) = \theta\bigl( [\theta(w_1)] \bigr) = \theta\bigl( [\theta(w_2)\cdot h] \bigr) = \theta\bigl([\theta(w_2)]\bigr)  = \theta(w_2),
     \end{equation*}
     where $[\theta(w_j)] \in \Gamma \backslash G /H$ denotes the equivalence class of $\theta(w_j)$.
     Since $H$ acts freely on $\Gamma \backslash G$, this implies that $h = 1$, as desired.

     To prove SL3, choose an open subset $Y \subseteq H$ around $1$ such that $\theta(W)\cdot Y$ is an open subset of $U \subseteq \Gamma \backslash G$.
     Condition SL3 for $\Sigma_x$, $\mathrm{i.e.}$ that the restriction of the right action $\rho$ yields an embedding $\Sigma_x \times Y \rightarrow X$ with open image, now follows from the following commutative diagram
     \begin{equation*}
         \xymatrix{ \Sigma_x \times Y \ar@{=}[r] \ar[d]_\rho & S_x \times \chi(\theta(W)) \times Y \ar[r]_\cong \ar[d]_\rho & S_x \times \chi'(\theta(W)\cdot Y) \ar[d]_\rho \\
         X \ar@{=}[r] & X \ar@{=}[r] & X, }
     \end{equation*}
     the fact that the right-most vertical map has this property by Remark \ref{rem:Independence-of-sections}, and that $\theta(W) \cdot Y$ is open in $\Gamma \backslash G$.
 \end{proof}
 \begin{lemma}\label{Lemma - Observer Moduli Space Fibre bundle}
    The canonical projection $\mathrm{Riem}^C(M) \rightarrow \ObsModuli[M][C]$ is a $\Diff_{x_0}(M)$-principal bundle.
\end{lemma}
\begin{proof}
    The group $\Diff_{x_0}(M)$ acts freely and properly on $\Riem[C](M)$, see \cite{Ebin1968Slice} for the latter.
    It remains to show that around each point $[g] \in \ObsModuli[M][C]$ exists a local section of the canonical projection.

    Since $\mathrm{Riem}^C(M)$ is an open subspace of the space of all Riemannian metrics, Ebin's slice theorem \cite{Ebin1968Slice} also applies to this subset.
    Thus, for each Riemannian metric $g\in \mathrm{Riem}^C(M)$, there is a \emph{slice} $S_g \subseteq \mathrm{Riem}^C(M)$ for the diffeomorphism group $\Diff(M)$.

    By Lemma \ref{Lemma - Frame bundle as diffeo quotient} we know that there is a open subset $V \subseteq \mathrm{Fr}^+(M) \cong \Diff(M)/\Diff_{x_0}(M)$ around the identity coset for which there is a section $\sigma \colon V \rightarrow \Diff(M)$.

    The stabiliser of a Riemannian metric $g$ in $\Diff(M)$ is the subgroup $\mathrm{Iso}(g)$ of all (orientation preserving) isometries, so the left action of $\mathrm{Iso}(g) \curvearrowright \Diff(M)/\Diff_{x_0}(M) \cong \mathrm{Fr}^+(M)$ is free.
    Indeed, an isometry $\varphi$ that fixes a frame $\mathbf{f} \in \mathrm{Fr}^+_x(M)$ acts as the identity on $T_xM$ and, therefore, it also acts as the identity in a geodesic neighbourhood around $x$.
    An open and closed argument now yields that $\varphi$ is the identity.

    Since all conditions of Lemma \ref{lem - Slices for good subgroups} are satisfied, we conclude that there is a slice $\Sigma_g$ around $g$ for the observer diffeomorphism group $\Diff_{x_0}(M)$ as well.

    The three slice conditions imply that $\Sigma_g /\mathrm{Stab}_{\Diff_{x_0}(M)}(g) \rightarrow \ObsModuli[M][C]$ is an embedding with an open image. 
    Thus, under this identification, the restriction of the map from SL3 to $\Sigma_g \times \{1\} = \Sigma_g$ yields the local section we are looking for
    \begin{equation*}
        \ObsModuli[M][C] \supseteq \Sigma_g/\mathrm{Stab}_{\Diff_{x_0}(M)}(g) = \Sigma_g  \xrightarrow{\mu} \mathrm{Riem}^C(M),
    \end{equation*}
    and $\Riem[C](M) \rightarrow \ObsModuli[M][C]$ is a $\Diff_{x_0}(M)$-principal bundle as claimed.
\end{proof}
 The following proposition list properties of $\HObsModuli[M][C]$ that are essential for our study of observer moduli spaces.

 \begin{proposition}\label{Prop - HModuliProperties}
 $ $
 \begin{itemize}
       \item[(1)] The diagram 
       \begin{equation}\label{Eq: HObsMod as pullback}
           \xymatrix{ \HObsModuli[M][C] \ar[r] \ar[d] & \HModuli[M][C] \ar[d] \\
           B\Diff_{x_0}(M) \ar[r] & B\Diff(M) }
       \end{equation}
       is a homotopy pullback.
       \item[(2)]
       The homotopy fibre of $\HObsModuli[M][C] \rightarrow \HModuli[M][C]$ is weakly homotopy equivalent to $\mathrm{Fr}^+(M)$.
       \item[(3)] The comparison map $\HObsModuli[M][C] \rightarrow \ObsModuli[M][C]$ is a weak homotopy equivalence.
   \end{itemize}
 \end{proposition}
 \begin{proof}
   We start with the proof of (1):
   The upper horizontal and the left vertical map yield a map from $\HObsModuli[M][C]$ into the pullback $E$ of $B\Diff_{x_0}(M) \rightarrow B\Diff(M) \leftarrow \HModuli[M][C]$, which itself is also a $\Riem[C](M)$-fibre bundle over $B\Diff_{x_0}(M)$.
   Thus, we obtain a map of fibre-bundles
   \begin{equation*}
       \xymatrix{ \mathrm{Riem}^C(M) \ar[r] \ar@{=}[d] & \HObsModuli[M][C] \ar[d] \ar[r] & B\Diff_{x_0}(M)\, \ar@{=}[d] \\
       \mathrm{Riem}^C(M) \ar[r] & E \ar[r] & B\Diff_{x_0}(M). }
   \end{equation*}
   The Five Lemma applied to the induced long exact sequences on homotopy groups now shows that the canonical map $\HObsModuli[M][C] \rightarrow E$ is a weak homotopy equivalence. 
   Since $\HModuli[M][C] \rightarrow B\Diff(M)$ is a fibration, the pullback of $B\Diff_{x_0}(M) \rightarrow B\Diff(M) \leftarrow \HModuli[M][C]$ agrees with the homotopy pullback, see \cite{Hirschhorn2003ModelCategories}*{Corollary 13.3.8}, and the claim follows.

   For the proof of (2), we choose the models of $B\Diff_{x_0}(M) = \mathrm{Emb}(M;\R^\infty)/\Diff_{x_0}(M)$ and $B\Diff(M) = \mathrm{Emb}(M;\R^\infty)/\Diff(M)$.
   In this way, the canonical map $B\Diff_{x_0}(M) \rightarrow B\Diff(M)$ is a numerable fibre bundle with fibre $\Diff(M)/\Diff_{x_0}(M)$, and similarly, we exhibit a fibre bundle 
   \begin{equation*}
       \xymatrix{\Diff(M)/\Diff_{x_0}(M) \ar[r] &  \HObsModuli[M][C] \ar[r] & \HModuli[M][C], }
   \end{equation*}
   from which we deduce that $\Diff(M)/\Diff_{x_0}(M)$ and the homotopy fibre of the canonical map $ \HObsModuli[M][C] \rightarrow \HModuli[M][C]$ are weakly homotopy equivalent.
   The statement follows now from Lemma \ref{Lemma - Frame bundle as diffeo quotient}.

   For the proof of (3), we argue as follows: Since $\Diff_{x_0}(M)$ is a closed subgroup of $\Diff(M)$ and since $\Diff(M)$ acts properly on $\Riem[C](M)$, so does $\Diff_{x_0}(M)$.
   Moreover, we have established that $\Diff_{x_0}(M) \curvearrowright \Riem[C](M)$ is free and that $\Riem[C](M) \rightarrow \HObsModuli[M][C]$ has a local section, see Lemma \ref{Lemma - Observer Moduli Space Fibre bundle}.
   Thus, Lemma \ref{lem - HQuot vs Quot} implies that the canonical comparison map $\HObsModuli[M][C] \rightarrow \ObsModuli[M][C]$ is a weak homotopy equivalence.
 \end{proof}
  
 By the universal property of a homotopy pullback, see \cite{dwyer1995homotopy}*{Prop 10.12} together with §8 in loc cit for the (Quillen) model structure on topological spaces, for every (retract of a) CW-complex $Y$, the set of homotopy classes $[Y,\HObsModuli[M][C]]$ is in one to one correspondence to homotopy classes of maps of commutative diagrams of the form
 \begin{equation*}
     \xymatrix@R-1em{& Y \ar[dl] \ar[d] \ar[dr]& \\
     B\mathrm{Diff}_{x_0}(M) \ar[r] & B\Diff(M) & \HModuli[M][C]. \ar[l] }
 \end{equation*}
 Using the classifying properties of the three spaces in this diagram, the task of construction a commutative diagram of this form translates into the construction of an $M$-fibre bundle $E$ over $Y$ that carries a Riemannian metric with curvature condition $C$ on its vertical tangent bundle, and whose structure group reduces to $\mathrm{Diff}_{x_0}(M)$.
 Thus, instead of constructing non-contractible maps into $\HObsModuli[M][C]$ directly, we will construct topologically non-trivial fibre bundles $E$ over spheres together with fibre metrics satisfying $C$ and show that their structure group reduces to $\Diff_{x_0}(M)$.



\section{The Construction of the bundle $\mathcal{DO}(1)$}\label{Section - Bundle Construction}

 We begin with introducing notation and conventions. 
 Sometimes, it will be convenient to identify $\C^2$ with the quaternions $\Quat$.
 We will always do this using the map $(z_0,z_1) \mapsto z_0 \cdot 1 + z_1 \cdot \J$, which is $\C$-linear if $\C$ acts on $\Quat$ from the left.

 \begin{definition}\label{Def - Sections of Hopf fibration}
     For $j = 0,1$, set $U_j = \{ [a_0:a_1] \in \CP^1 \,:\, a_j \neq 0 \}$. Define using quaternions
     \begin{align*}
         \sigma_0 \colon U_0 &\rightarrow S^3 \subseteq \mathbb{H}, \quad [a_0:a_1] \mapsto \frac{1 + a_1/a_0 \cdot \J}{|1 + a_1/a_0 \cdot \J|}, \\
         \sigma_1 \colon U_1 &\rightarrow S^3 \subseteq \mathbb{H}, \quad [a_0:a_1] \mapsto \frac{a_0/a_1 + \J}{|a_0/a_1 + \J|}.
     \end{align*}
 \end{definition}
 Observe that $\sigma_j$ are local sections of the Hopf fibration $S^3 \rightarrow \CP^1 = S^1\backslash S^3$, so that $z\sigma_j([z])^{-1} \in S^1$ if $[z]\in U_j$.
 In particular, from $\sigma_1([a_0:a_1]) = a_0/a_1 \cdot |a_0/a_1|^{-1} \cdot \sigma_0([a_0:a_1])$ we deduce that
 \begin{equation}
     \sigma_0([a_0:a_1]) \cdot \sigma_1([a_0:a_1])^{-1} = \frac{a_1/a_0}{|a_1/a_0|} \in S^1 \subseteq \C
 \end{equation}
 takes values in the complex numbers.

 \begin{definition}\label{Def - Twisted Hopffibration}
    Define the $\CP^1$-parametrised Hopf-fibration
    \begin{equation*}
        \Phi \colon S^3 \times \CP^1 \rightarrow \CP^1 \times \CP^1 \quad \text{via} \quad (h,[a_0:a_1]) \mapsto \bigl( [(a_0 + a_1\J)/ |a_0 + a_1\J| \cdot h], [a_0:a_1] \bigr).
    \end{equation*}
 \end{definition}

 \begin{lemma}\label{Lemma - Twisted Hopf S1 Principal Bundle}
     The map $\Phi \colon S^3 \times \CP^1 \rightarrow \CP^1 \times \CP^1$ is an $S^1$-principal bundle. 
 \end{lemma}
 \begin{proof}
     The fibre of $\Phi$ at the point $\bigl([b_0:b_1],[a_0:a_1]\bigr)$ can be identified with
     \begin{align*}
         \left\{ h \in S^3 \, : \, h = \frac{(a_0 + a_1\J)^{-1}}{|(a_0 + a_1\J)^{-1}|}\mathrm{e}^{\I \theta}\frac{(b_0 + b_1\J)}{|b_0 + b_1\J|}, \theta \in \R \right\} \cong S^1.
     \end{align*}
     Indeed, an element $(h,[c_0:c_1])$ is mapped by $\Phi$ to $\bigl([b_0:b_1],[a_0:a_1]\bigr)$ if and only if $[c_0:c_1] = [a_0: a_1]$ and if there is a complex phase $\mathrm{e}^{\I \theta} \in S^1$ such that 
     \begin{equation*}
         \frac{b_0 + b_1\J}{|b_0 + b_1\J|}  = \mathrm{e}^{\I \theta}\cdot \frac{a_0 + a_1\J}{|a_0 + a_1\J|} \cdot h. 
     \end{equation*}
     Rearranging this formula yields the claimed description of the preimage of $\Phi$.
          
     We define the right-action of $S^1$ on the fibre by
     \begin{equation*}
         (h,[a_0:a_1])_\bullet \mathrm{e}^{\I \theta} = \left(\mathrm{conj}\bigl((a_0 + a_1\J)^{-1}\bigr)(\mathrm{e}^{\I \theta}) \cdot h, [a_0:a_1]\right),
     \end{equation*}
     where $\mathrm{conj}(q)(x) = qxq^{-1}$ denotes the conjugation with unit quaterions.
     This formula defines a smooth, free, proper, and fibre-preserving action on $S^3 \times \CP^1$, which exhibits $\Phi \colon S^3 \times \CP^1 \rightarrow \CP^1 \times \CP^1$ as an $S^1$-principal bundle.
 \end{proof}
 
 \begin{lemma}\label{Lemma - ChernClass Twisted Hopf-Fibration}
    Under the identification $H^\ast(\CP^1 \times \CP^1) \cong \Z\lbrack x,y\rbrack/\langle x^2,y^2\rangle$ with $x = \mathrm{pr}_1^\ast(c)$, $y = \mathrm{pr}_2^\ast(c)$, where $c \in H^2(\CP^1)$ is the Chern class of the tautological line bundle, the first Chern class of $\Phi$ satisfies 
    \begin{equation*}
        c_1(\Phi) = x - y.
    \end{equation*}
 \end{lemma}
 \begin{proof}
     Let $f_\Phi \colon \CP^1 \times \CP^1 \rightarrow \CP^\infty$ be the classifying map of $\Phi$.
     By the universal coefficient theorem, we may determine $c_1(\Phi) = H^2(f_\Phi)(c_1^{\mathrm{univ}})$ by determining $H_2(f_\Phi)$, where $c_1^{\mathrm{univ}} \in H^2(\CP^\infty;\Z)$ is the (universal) first Chern class.
     The diagram
     \begin{equation*}
         \xymatrix{ S^3 \times \CP^1 \ar[rr] \ar[d]_\Phi && S^{2\infty - 1} \simeq \{\mathrm{pt}\} \ar[d] \\
         \CP^1 \times \CP^1 \ar[rr]^{f_\Phi} && \CP^\infty }
     \end{equation*}
     commutes, so $\mathrm{im}(H_2(\Phi)) \subseteq \ker(H_2(f_\Phi))$.
     Since $H_2(S^3 \times \CP^1)$ is generated by $1\times [\CP^1]$ and since $\Phi$ restricts on $1 \times \CP^1$ to the diagonal map $\Delta \colon \CP^1 \rightarrow \CP^1 \times \CP^1$, we deduce  $\mathrm{im}(H_2(\Phi)) = \Z \cdot \left( [\CP^1] \times 1 + 1\times [\CP^1] \right)$.
     
     The restriction of $\Phi$ to $S^3\times\{[1:0]\}$ agrees with the classical Hopf-fibration, so the restriction $f_\Phi|_{\CP^1 \times 1}$ is homotopic to the classifying map of the Hopf-fibration, which is the canonical inclusion $\CP^1 \hookrightarrow \CP^\infty$.
     This implies that $H_2(f_\Phi)([\CP^1 \times 1]) = [\CP^1] \in H_2(\CP^\infty)$ and whence $H_2(f_\Phi)([1 \times \CP^1]) = -[\CP^1]$.

     The universal coefficient theorem now implies $c_1(\Phi) = x - y$.
 \end{proof}

 Recall that the disc bundle $D\mathcal{O}(1)$ of the dual of the tautological line bundle can be described as a Borel construction using the Hopf-fibration and the representation $(\placeholder)^{-1} \colon S^1 \curvearrowright \C$ that is given by $(\mathrm{e}^{\I\theta},\lambda) \mapsto \mathrm{e}^{-\I \theta}\lambda$.
 In formulas, we have
 \begin{equation*}
     D\mathcal{O}(1) = S^3 \times_{S^1,(\cdot)^{-1}} D^2 = \left( S^3 \times D^2 \right)/\sim
 \end{equation*}
 with $(z,\zeta) \sim (\mathrm{e}^{\I \theta}z, \mathrm{e}^{\I \theta} \zeta)$.
 In this light, we define 
 \begin{equation*}
     \mathcal{DO}(1) := (S^3 \times \CP^1) \times_{\Phi,(\cdot)^{-1}} D^2 = \left(S^3 \times \CP^1 \times D^2\right)/\sim 
 \end{equation*}
 with $\left( z, [a_0:a_1], \zeta \right) \sim \left( (z,[a_0:a_1])_\bullet \mathrm{e}^{\I\theta}, \mathrm{e}^{\I \theta} \zeta \right)$ for all $\theta \in \R$.
 This space should be thought of as a bundle of disc bundles $D\mathcal{O}(1)$ over $\CP^1$.

 \begin{lemma}\label{Lemma - TransTwistedDiscBundle}
     The map $p \colon \mathcal{DO}(1) \rightarrow \CP^1$ induced by the projection to $\CP^1$ induces a fibre-bundle structure on the total space with fibre $D\mathcal{O}(1)$.
     Its structure group can be reduced to $\Diff(D\mathcal{O}(1),S^3)$, the group of all diffeomorphisms that agree with the identity near the boundary $S^3$.
 \end{lemma}
 \begin{proof}
     Local trivialisations of this bundle are given by
     \begin{align*}
         \Psi_j \colon p^{-1}(U_j) &\rightarrow D\mathcal{O}(1) \times U_j = (S^3 \times_{S^1,(\cdot)^{-1}} D^2) \times U_j, \\
         [h,[a_0:a_1], \zeta] &\mapsto \left( [\sigma_j([a_0:a_1])\cdot h, \zeta] , [a_0:a_1]\right).
     \end{align*}
     The corresponding transition function is $g_{01}([a_0:a_1])([z,\zeta]) = [\sigma_0\cdot \sigma_1^{-1}([a_0:a_1]) \cdot z, \zeta]$.
     
     The map $g_{01} \colon U_0\cap U_1 \rightarrow \Diff(D\mathcal{O}(1))$ factors through $S^1 \subseteq \C^\times$, so considered as a map with values in $S^3$ it is null homotopic.
     Pick a smooth homotopy $H \colon U_0 \cap U_1 \times [1/2,1] \rightarrow S^3$ such that $H(\placeholder,t) = g_{01}$ for all $t \in [4/8,5/8]$ and $H(\placeholder,t) = 1_{S^3}$ if $t \in [7/8,1]$.
     Define $\Tilde{g}_{01} \colon U_0 \cap U_1 \rightarrow \Diff(D\mathcal{O}(1),S^3)$ via
     \begin{equation*}
         \Tilde{g}_{01}([a_0:a_1])([z,\zeta]) = \begin{cases}
             g_{01}([a_0:a_1])([z,\zeta]), & \text{ if } |\zeta| \leq 1/2, \\
             \bigl[H([a_0:a_1],|\zeta|) \cdot \Bar{\zeta}/|\zeta| \cdot z, |\zeta|\bigr], & \text{ if } |\zeta| \geq 1/2.
         \end{cases}
     \end{equation*}
     This smooth transition function is homotopic to $g_{01}$.
     In fact, one can choose a homotopy $\mathcal{H}$ that agrees with $g_{01}$ on the subspace $\CP^1 \times (U_0 \cap U_1) \subseteq D\mathcal{O}(1) \times (U_0 \cap U_1)$.

     Via the clutching construction, as for example explained in \cite{Husemoller1994FibreBundles}*{Prop 7.1}, the homotopy $\mathcal{H}$ yields a $D\mathcal{O}(1)$-fibre bundle $E_{\mathcal{H}} \rightarrow \CP^1 \times [0,1]$ that restricts to $\mathcal{DO}(1)$ on $\CP^1 \times \{0\}$ and a fibre bundle with a $\Diff(D\mathcal{O}(1),S^3)$ reduction on $\CP^1 \times \{1\}$.
     By the homotopy theorem for fibre bundles, see  \cite{tomDieck2008algebraic}*{Theorem 14.3.2}, the bundle $\mathcal{DO}(1)$ is isomorphic to $E_{\mathcal{H}}|_{\CP^1 \times \{1\}}$, so it has a $\Diff(D\mathcal{O}(1),S^3)$ reduction.
 \end{proof}
 \begin{rem}
     A word of warning:
     $H([a_0:a_1],|\zeta|) \in S^3$ is not $\C^\ast$-homogeneous, so the map $\left( [a_0:a_1],[z,\zeta] \right) \mapsto [H([a_0:a_1],|\zeta|) \cdot \Bar{\zeta}/|\zeta| \cdot z, |\zeta|]$ is only well defined on the subspace $D\mathcal{O}(1)\setminus \CP^1 \times U_0 \cap U_1$ (i.e. only away from the zero section) and does not extend to $\CP^1 \times U_0 \cap U_1$.
     In particular, the fact that every map $U_0 \cap U_1 \rightarrow S^3$ is null-homotopic cannot be used to prove that $\mathcal{DO}(1)$ is a trivial bundle over $\CP^1$.
     In fact, this is false by Lemma \ref{Lemma - CharClassTwistedTaut} below.     
 \end{rem}

 \begin{rem}\label{rem - Trivial Boundary of Disc Bundle}
     Lemma \ref{Lemma - TransTwistedDiscBundle} implies in particular that the boundary of the total space $\mathcal{DO}(1)$ is a trivial $S^3$-fibre bundle over $\CP^1$ because the structure group reduces to $\mathrm{Diff}(D\mathcal{O}(1),S^3)$.
     An explicit trivialisation is given by 
     \begin{equation*}
         S^3 \times \CP^1 \rightarrow \partial \mathcal{DO}(1) \qquad (p,[a_0:a_1]) \mapsto [p,[a_0:a_1],1].
     \end{equation*}
     We will use this trivialisation later to glue in this twisted disc bundle to a product $W^4  \times \CP^1$, where $W^4$ is a compact manifold with boundary $S^3$, along their common boundary $S^3 \times \CP^1$.
 \end{rem}

 Although we will mostly consider $\mathcal{DO}(1)$ as bundle over $\CP^1$ with fibre $D\mathcal{O}(1)$, it can also be considered as a disc bundle over $\CP^1 \times \CP^1$ using the projection $\mathcal{DO}(1) \rightarrow \CP^1 \times \CP^1$ that is given by $[p,[q],\lambda] \mapsto \Phi([p,[q]]) = ([qp],[q])$.
 Hence, there is a fibrewise zero section $\CP^1 \times \CP^1 \hookrightarrow \mathcal{DO}(1)$ that is given by $([p],[q]) \mapsto [q^{-1}p, [q], 0]$.
 Obviously, the composition of this map with the projection $p$ from Lemma \ref{Lemma - TransTwistedDiscBundle} is the projection to the second component, the base $\CP^1$. 

 \begin{lemma}\label{Lemma - CharClassTwistedTaut}
    With the notation from Lemma \ref{Lemma - ChernClass Twisted Hopf-Fibration} the total Stiefel-Whitney class and the first Pontrjagin class of the vertical tangent bundle satisfy the following identities:
    \begin{align*}
        w_\bullet(T^\vertical \mathcal{DO}(1)) &= 1 - (x-y) \in H^\ast(\CP^1 \times \CP^1;\Z_2), \\
        p_1(T^\vertical \mathcal{DO}(1)) &= -2xy.
    \end{align*}
 \end{lemma}
 \begin{proof}
    Since the fibrewise zero section $\CP^1 \times \CP^1 \hookrightarrow \mathcal{DO}(1)$ that is given by $([p],[q]) \mapsto [q^{-1}p,[q],0]$ is a homotopy equivalence, it is enough to determine the total Stiefel-Whitney class and the first Pontrjagin class of $T^\vertical \mathcal{DO}(1)|_{\CP^1 \times \CP^1}$.
    
    As $\mathcal{DO}(1)$ is a disc bundle over $\CP^1 \times \CP^1$, $\mathrm{i.e.}$ the bundle $(S^3 \times \CP^1) \times_{\Phi, (\placeholder)^{-1}} \C$, the complex line bundle associated to the $S^1$-principal bundle of Definition \ref{Def - Twisted Hopffibration} using the $S^1$ representation $(\placeholder)^{-1} \colon S^1 \curvearrowright \C$ given by $(\mathrm{e}^{\I \theta},\lambda) \mapsto \mathrm{e^{-\I \theta}}\lambda$, serves as a normal bundle of $\CP^1 \times \CP^1$ inside $\mathcal{DO}(1)$.
    Indeed, using the equivalence classes of curves to describe tangent vectors, this isomorphism is given by
    \begin{equation*}
        (S^3 \times \CP^1) \times_{\Phi, (\placeholder)^{-1}} \C \rightarrow \nu(\CP^1 \times \CP^1) \subseteq T\mathcal{DO}(1), \qquad [p,[q],\lambda] \mapsto \bigl[ t \mapsto [p,[q],t\lambda] \bigr].
    \end{equation*}
    This normal bundle is, in fact, a subbundle of $T^\vertical \mathcal{DO}(1)|_{\CP^1 \times \CP^1}$.
    Since the fibrewise zero section is a map over the base $\CP^1$, we get the following decomposition of the restricted vertical tangent bundle
        \begin{align*}
           T^\vertical \mathcal{DO}(1)|_{\CP^1 \times \CP^1} &= T^\vertical(\CP^1 \times \CP^1) \oplus (S^3 \times \CP^1) \times_{\Phi,(\cdot)^{-1}} \C \\
           &= \mathrm{pr}_1^\ast T\CP^1 \oplus (S^3 \times \CP^1) \times_{\Phi,(\cdot)^{-1}} \C,
       \end{align*}
    where $T^\vertical(\CP^1 \times \CP^1)$ is the kernel of differential of the projection to the second component.   

    It is well known that the characteristic class of a (complex) vector bundle agrees with the characteristic class of its associated principal bundle of (unitary) frames.
    The bundle $(S^3 \times \CP^1) \times_{\Phi,(\placeholder)^{-1}} \C$ is the complex conjugate bundle of $(S^3 \times \CP^1) \times_{\Phi,(\placeholder)^{+1}} \C$, and the associated bundle of unitary frames agrees with the $S^1$-principal bundle $\Phi \colon S^3 \times \CP^1 \rightarrow \CP^1 \times \CP^1$.
    Hence, we deduce from Lemma \ref{Lemma - ChernClass Twisted Hopf-Fibration} that 
    \begin{equation*}
        c_1\bigl( (S^3 \times \CP^1) \times_{\Phi,(\placeholder)^{-1}} \C \bigr) = - c_1\bigl((S^3 \times \CP^1) \times_{\Phi,(\placeholder)^{+1}} \C\bigr) = -c_1(\Phi). 
    \end{equation*}
    The decomposition yields, together with the fact that the second Stiefel-Whitney class is the mod $2$ reduction of the first Chern class, that
         \begin{align*}
             w_\bullet( T^\vertical \mathcal{DO}(1)|_{\CP^1 \times \CP^1} ) &= w_\bullet(\mathrm{pr}_1^\ast T \CP^1) \cdot w_\bullet\bigl((S^3 \times \CP^1) \times_{\Phi,(\cdot)^{-1}} \C\bigr) \\ 
             &= 1 \cdot (1 - w_2(\Phi)) \\
             &= 1 - (x-y). 
         \end{align*}
    Because $H^4(\CP^1 \times \CP^1)$ is torsion free, we deduce from the Whitney sum formula and the expression of the Pontrjagin classes in terms of Chern classes, see \cite{milnor1974characteristic} for both, that
         \begin{align*}
             p_1(T^\vertical \mathcal{DO}(1)|_{\CP^1 \times \CP^1}) &= \mathrm{pr}_1^\ast \bigl(p_1(T\CP^1)\bigr) +  p_1\bigl((S^3 \times \CP^1) \times_{\Phi,(\cdot)^{-1}} \C\bigr) = c_1(\Phi)^2 = -2xy.
         \end{align*}
 \end{proof}

 \begin{rem}
     After the first version of this paper had been written, the author was informed that Jiafeng Lin in \cite{lin2022family} studied a similar family, which in our notation would be the bundle $\mathcal{DO}(-1) \rightarrow \CP^1 \xrightarrow{\bar{\cdot}} \CP^1$.
     The complex conjugation $\bar{\cdot}$ arises from the different conventions of parametrising the complex structures, which also causes different signs in our results.
 \end{rem}


\section{Proof of Theorem \ref{Main Thm - PSC} and \ref{Main Thm - PSC with order}}\label{Section - Positive Scalar Curvature }

We wish to construct a fibrewise positive scalar curvature metric on $\mathcal{DO}(1)$ that has product structure near the boundary.
To this end, recall, for example from \cite{EbertFrenck2021Surgery}, that a \emph{torpedo metric} on $D^2$ is a Riemannian metric $\gtorpedo$ that is $\Or[2]$-invariant (with respect to the tautological action), has a product structure near the boundary $S^1$, and non-negative scalar curvature everywhere. 
For the sake of concreteness, we may assume that $\gtorpedo$ is the product metric  $\diff r^2 \oplus g_{S^1}$ on $D^2\setminus 2/3\cdot D^2$.
For later convenience, we will consider the entire family of metrics $\gtorpedo[,\lambda]$ that agree on $D^2\setminus 2/3D^2$ with $\diff r^2 \oplus \lambda^2 g_{S^1}$ for all $\lambda \in \R_{>0}$. 

Recall that a \emph{fibrewise Riemannian submersion} is a bundle map $F \colon (E_1,g_1) \rightarrow (E_2,g_2)$ between two fibre bundles over the same base space $B$ that covers the identity of $B$ such that on each fibre its differential $TF_b \colon TE_{1,b} \rightarrow TE_{2,b}$ is surjective and such that its restriction to the orthogonal complement of $\ker TF_b$ with respect to $g_{1,b}$ is an isometry.

If we equip each fibre of $S^3\times \CP^1 \times D^2 \rightarrow \CP^1$ with the product metric of the round sphere $\ground$ and a torpedo metric $\gtorpedo[,\lambda]$, then the circle action that was used to define $\mathcal{DO}(1)$ is isometric.
This observation allows us to prove the next lemma. 

\begin{lemma}\label{Lemma - fibrewise Riem submersion}
    There is a unique fibrewise Riemannian metric $\gDO[,\lambda]$ on $\mathcal{DO}(1)$ such that the canonical projection
    \begin{equation*}
        (S^3 \times \CP^1 \times D^2, \ground \oplus \gtorpedo[,\lambda]) \rightarrow \left((S^3 \times \CP^1) \times_{\Phi,(\cdot)^{-1}} D^2 = \mathcal{DO}(1),\gDO[,\lambda]\right)
    \end{equation*}
    is a fibrewise Riemannian submersion.
    Moreover, $\gDO[,\lambda]$ is a fibrewise positive scalar curvature metric.
\end{lemma}
\begin{proof}
    As the action is isometric, a point-wise application of \cite{Besse1987EinsteinMfds}*{Formula 9.12} shows that the canonical projection is a fibrewise Riemannian submersion.
    
    O'Neills formula \cite{ONeill1966SubmersionEq}*{Corollary 1} yields that all sectional curvatures are non-decreasing, which implies that $(\mathcal{DO}(1),\gDO[,\lambda])$ has non-negative sectional curvature as well.

    To see that the scalar curvature of $\gDO[,\lambda]$ is positive, we argue as follows:
    Fix a point $(p,[q],\lambda) \in S^3 \times \CP^1 \times D^2$.
    The tangent space of the orbit of the $S^1$-action presented in the proof of Lemma \ref{Lemma - Twisted Hopf S1 Principal Bundle} yields a one-dimensional subbundle $K \subseteq T^\vertical (S^3 \times \CP^1 \times D^2) = \bigl(T(S^3 \times D^2)\bigr) \times \CP^1 $, so we find a two-dimensional plane $\Pi_{(p,[q])} \subseteq T^\vertical_{(p,[q])}(S^3 \times \CP^1) = T_p S^3 \subseteq T^\vertical_{(p,[q],z)} (S^3 \times \CP^1 \times D^2)$ that is perpendicular to $K_{(p,[q])}$ (with respect to the product of the round metric and the flat metric). 
    The sectional curvature $\mathrm{sec}(\Pi_{(p,[q])})$ is positive, and hence the ($2$-dimensional) image of this plane inside $T^\vertical_{[p,[q],\lambda]} \mathcal{DO}(1)$ still has positive sectional curvature by \cite{ONeill1966SubmersionEq}*{Corollary 1}.
    Thus, the scalar curvature of $\gDO[,\lambda]$ is positive.
\end{proof}

We need to study its behaviour near the boundary.

\begin{lemma}
   Under the collar map 
   \begin{align*}
       \kappa \colon (2/3,1] \times S^3 \times \CP^1 &\rightarrow  S^3 \times \CP^1 \times_{\Phi,(\cdot)^{-1}} D^2 = \mathcal{DO}(1), \\
       (r,z,[a_0:a_1]) &\mapsto [z,[a_0:a_1],r\cdot 1],
   \end{align*}
   the quadratic form of the fibrewise metric pulls back to
   \begin{equation*}
       (\kappa^\ast \gDO[,\lambda])_{(r,z,[a_0:a_1])} = \diff r^2 \oplus \ground - \frac{\ground\bigl( \,\cdot \, , (|a_0|^2-|a_1|^2- 2\bar{a}_0a_1\J) \I \cdot z\bigr)^2}{(1+\lambda^2)\cdot (|a_0|^2 + |a_1|^2)^2}.
   \end{equation*}
\end{lemma}
\begin{proof}
    The collar decomposes into
    \begin{equation*}
        \xymatrix{ (2/3,1] \times S^3 \times \CP^1 \ar@/_1pc/[rr]_-{\kappa} \ar[r]^-{\Tilde{\kappa}} & S^3 \times \CP^1 \times D^2 \ar[r]^-p & \mathcal{DO}(1),  }
    \end{equation*}
    where $p$ is the canonical projection.
    The kernel of $Tp$ is the tangent bundle of $S^1$-orbits of all $S^1$-action that is used to define $\mathcal{DO}(1)$.
    In the following, we will refer to this subbundle as the \emph{infinitesimal orbit}.
    
    Recall that this action $S^1 \times \bigl(S^3 \times \CP^1 \times D^2\bigr) \rightarrow S^3 \times \CP^1 \times D^2$ is given by
    \begin{equation*}
         \bigl(\mathrm{e}^{\I \theta}, (z,[a_0:a_1],\lambda)\bigr) \mapsto \Bigl((a_0 + a_1\J)^{-1} \mathrm{e}^{\I \theta} (a_0 + a_1) \cdot z,[a_0:a_1],\mathrm{e}^{\I \theta}\lambda \Bigr). 
    \end{equation*}
    Differentiating by $\theta$ at the point $(z,[a_0:a_1],r1)$ yields the $1$-dimensional vector space 
    \begin{equation*}
        \bigl\{  \left( (a_0 + a_1\J)^{-1} \I \theta (a_0 + a_1\J) \cdot z, \I r\theta  \right) : \theta \in \R \bigr\},
    \end{equation*}
    so we can identify $T^\vertical_{(z,[a_0:a_1],\zeta)}\mathcal{DO}(1)$ with the orthogonal complement of this infinitesimal orbit inside $T_z S^3 \oplus T_r D^2$.

    The differential of $\Tilde{\kappa}$ is given by
    \begin{align*}
        D_{(r,1,[a_0:a_1])}\Tilde{\kappa} \colon\  \R \oplus T_1S^3 \times \{[a_0:a_1]\} &\rightarrow T_1 S^3 \oplus T_r D^2, \\
        (\partial_r, v) &\mapsto (v,1),
    \end{align*}
    and the projection $\mathfrak{p}$ to the orthogonal complement of the infinitesimal orbits at these points is given by the formula\footnote{We recall that $\mathrm{conj}(a)(b) = aba^{-1}$ does not refer to the quaternionic conjugation.} 
    \begin{align*}
        &\quad \mathfrak{p}_{(z,[a_0:a_1],r)}(v,\zeta) \\
        &= (v,\zeta) - \frac{\ground \oplus \gtorpedo[,\lambda]\bigl( (v,\zeta),(\mathrm{conj}((a_0 + a_1\J)^{-1})(\I\theta)\cdot z,\I r\theta  ) \bigr) \cdot \bigl( \mathrm{conj}((a_0 + a_1\J)^{-1})(\I\theta)\cdot z,\I r\theta  \bigr)}{|| \bigl(\mathrm{conj}((a_0 + a_1\J)^{-1})(\I \theta),\I r\theta  \bigr) ||_{\ground \oplus \gtorpedo[,\lambda]}^2}\\
        &= (v,\zeta) - \left( \ground(v,\mathrm{conj}((a_0 + a_1\J)^{-1})(\I \theta)\cdot z) + \gtorpedo[,\lambda](\zeta,\I r\theta) \right) \cdot \frac{\bigl(\mathrm{conj}((a_0 + a_1\J)^{-1})(\I\theta) \cdot z,\I r\theta  \bigr)}{(1+\lambda^2)|\theta|^2}.
    \end{align*}
    
    We calculate without problems that $\mathfrak{p}(0,1) = (0,1)$, which implies
    \begin{equation*}
        (\kappa^\ast \gDO[,\lambda])_{(r,z,[a_0:a_1])}(\partial_r,\partial_r) = (\ground \oplus \gtorpedo[,\lambda])_{(z,[a_0:a_1],r)}(\mathfrak{p}(0,1),\mathfrak{p}(0,1)) = 1, 
    \end{equation*}
    and
    \begin{equation*}
        (\kappa^\ast \gDO[,\lambda])_{(r,z,[a_0:a_1])}(v,\partial_r) = (\ground \oplus \gtorpedo[,\lambda])_{(z,[a_0:a_1],r)}(\mathfrak{p}(v,0),\mathfrak{p}(0,1)) = 0.
    \end{equation*}
    Finally, we have
    \begin{align*}
        (\kappa^\ast \gDO[,\lambda])_{(r,z,[a_0:a_1])}(v,v) &= (\ground \oplus \gtorpedo[,\lambda])_{(z,[a_0:a_1],r)}(\mathfrak{p}(v,0),\mathfrak{p}(v,0)) \\
        &= \ground(v,v) - \frac{\ground\bigl(v,\mathrm{conj}((a_0 + a_1\J)^{-1})(\I \theta) \cdot z\bigr)^2}{(1+\lambda^2)|\theta|^2} \\
        &= \ground(v,v) - \frac{\ground\bigl(v,(|a_0|^2-|a_1|^2- 2\bar{a}_0a_1\J) \I \cdot z \bigr)^2}{(1+\lambda^2)\cdot (|a_0|^2 + |a_1|^2)^2} 
    \end{align*}
    as claimed.
\end{proof}

To remove the dependency of $\kappa^\ast \gDO = \kappa^\ast \gDO[,1]$ from the base point $[a_0:a_1]$, we first observe that
\begin{equation}\label{eq: curve of boundary metric}
    [0,1] \ni t \mapsto \ground - t\cdot \frac{\ground\bigl(\, \cdot \, ,(|a_0|^2-|a_1|^2- 2\bar{a}_0a_1\J) \I \cdot z\bigr)^2}{2(|a_0|^2 + |a_1|^2)^2} 
\end{equation}
defines a smooth curve of fibrewise positive scalar curvature metrics on the bundle $(S^3 \times \CP^1)$ because $\kappa^\ast \gDO[,(2/t-1)^{1/2}]$ is a positive scalar curvature metric on $(2/3,1] \times (S^3 \times \CP^1)$ that interpolates between the constant family of round metrics on $S^3$ and the restriction of $\gDO$ to $S^3 \times \CP^1$. 

\begin{rem}
    The following alternative strategy to derive that the curve of metrics defined in (\ref{eq: curve of boundary metric}) converges to the round metric was communicated to the author by the anonymous referee.
    The torpedo metric $\gtorpedo[,\lambda]$ is a product metric near the boundary, so the fibre metric $\kappa^\ast\gDO[,\lambda]$ is a family of Cheeger deformations of the product metric $g_\circ \oplus \diff r^2$, where the $S^1$-action on the fibre $S^3 \times \{[q] \times D^2 \setminus 2/3 D^2\}$ acts on the $S^3$ factor by $p_\bullet \mathrm{e}^{\I\theta}  = q^{-1}\mathrm{e}^{\I\theta}qp$.
    It is a general fact that Cheeger deformations increase all sectional curvatures and that Cheeger deformations of metrics $\kappa^\ast\gDO[,\lambda]$ converge to the original metric $g_\circ \oplus \diff r^2$ for $\lambda \to \infty$, see \cite{Ziller2007SurveyNonNegCurv}*{p.4} for both statements. 
\end{rem}

Pick a smooth curve $\psi \colon [0,1] \rightarrow [0,1]$ that is monotonically increasing, locally constant in a neighbourhood of $\{0,1\}$ with $\psi(0) = 0$ and $\psi(1)=1$, and satisfies $\psi' \leq 1.2$ as well as $|\psi''| \leq 10$.
By \cite{bar2005generalized}*{Formula 4.8}, we can choose a sufficient large $R > 0$ such that
\begin{equation*}
    \Tilde{g}_{(r,z,[a_0:a_1])} = \diff r^2 \oplus \ground - (1-\psi((r-1)/R)) \cdot \frac{ \ground\bigl(\, \cdot \, ,(|a_0|^2-|a_1|^2- 2\bar{a}_0a_1\J) \I \cdot z\bigr)^2}{2(|a_0|^2 + |a_1|^2)^2}
\end{equation*}
is a fibrewise psc metric on $[1,R+1] \times S^3 \times \CP^1$ so that $\gDO \cup \tilde{g}$ is a fibrewise psc metric on the fibre bundle
$\mathcal{DO}(1) \cup \left([1,R+1] \times S^3 \times \CP^1\right) \rightarrow \CP^1$
that is of the form $\diff r^2 \oplus \ground$ near the boundary.
This fibre bundle is of course fibre-diffeomorphic to $\mathcal{DO}(1)$ (via a reprametrisation of the collar) and we denote the pullback of $\gDO \cup \tilde{g}$ under this diffeomorphism with $\gDOext$.

\begin{proof}[Proof of Theorem \ref{Main Thm - PSC}]
    By the (parametrised) surgery theorem, see, for example, \cite{chernysh2004homotopy} or \cite{EbertFrenck2021Surgery}, the space $\Riem[\scal > 0](M,D^4)$ consisting of all positive scalar curvature metrics that restrict to a fixed torpedo metric $\gtorpedo$ on $D^4$ is a homotopy equivalent subspace of $\Riem[\scal > 0](M)$.
    Thus, we may assume that $g_0$ lies in this subspace to begin with.

    If we set $W := M \setminus 2/3\cdot D^4$, then $\CP^2 {\sharp}_\Phi M := \mathcal{DO}(1) \cup_{S^3 \times \CP^1} (W \times \CP^1)$ is a smooth fibre bundle with fibre $\CP^2 \sharp M$ whose structure group is $\Diff(\CP^2\sharp M, M \setminus 2/3D^4)$, the subgroup of all diffeomorphisms that restrict to the identity on $M \setminus 2/3D^4$.
    Of course, $\gDOext \cup g_0|_W$ is a fibrewise psc metric, so this bundle yields a classifying map 
    \begin{equation*}
       f_{\CP^2 {\sharp}_\Phi M,\gDOext \cup g_0|_W}\colon  S^2 = \CP^1 \rightarrow \HObsModuli[\CP^2\sharp M][\scal > 0]. 
    \end{equation*}

    We show using characteristic numbers that this bundle is non-trivial.
    To this end, observe that $\CP^1 \times \CP^1$ embedds into the interior of $\mathcal{DO}(1)$ so that 
    \begin{equation*}
        T^\vertical \CP^2 {\sharp}_\Phi M |_{\CP^1\times \CP^1} = T^\vertical \mathcal{DO}(1)|_{\CP^1\times \CP^1}.
    \end{equation*} 
    Denote this embedding by $\iota$.
    It follows from Lemma \ref{Lemma - CharClassTwistedTaut} that
    \begin{align*}
        \langle p_1(T^\vertical \CP^2 {\sharp}_\Phi M), \iota_\ast [\CP^1 \times \CP^1] \rangle & = \langle p_1(T^\vertical \CP^2 {\sharp}_\Phi M |_{\CP^1 \times \CP^1}), [\CP^1 \times \CP^1]  \rangle \\
        &=  \langle p_1(T^\vertical \mathcal{DO}(1) |_{\CP^1 \times \CP^1}), [\CP^1 \times \CP^1]  \rangle \\
        &= \langle -2xy, [\CP^1\times\CP^1] \rangle = -2.
    \end{align*}
    
    On the other hand, Hirzebruchs signature theorem together with the universal coefficient theorem implies that $p_1(N) \in \mathrm{Hom}(H_4(N;\Z),3\Z) \subseteq H^4(N;\Z)$ for every closed, oriented, smooth four-manifold $N$.
    Thus, a (global) trivialisation $\mathfrak{T}\colon \CP^2 {\sharp}_\Phi M \rightarrow (\CP^2 \sharp M) \times \CP^1$ would yield
    \begin{equation*}
        T^\vertical \CP^2 {\sharp}_\Phi M = \mathfrak{T}^\ast T^\vertical(\CP^2 \sharp M \times \CP^1) = \mathfrak{T}^\ast\left(\mathrm{pr}_1^\ast T\CP^2 \sharp M\right)
    \end{equation*}
    and we would deduce from
    \begin{align*}
        \langle p_1(T^\vertical \CP^2 {\sharp}_\Phi M), \iota_\ast [\CP^1 \times \CP^1] \rangle &= \langle \iota^\ast \mathfrak{T}^\ast \pr_1^\ast p_1(T\CP^2 \sharp M), [\CP^1 \times \CP^1]  \rangle \\
        &= \langle p_1(T\CP^2 \sharp M), (\mathrm{pr}_1 \circ \mathfrak{T} \circ \iota)_\ast [\CP^1 \times \CP^1] \rangle \in 3\Z 
    \end{align*}
    that $-2 \in 3\Z$, which is a contradiction.
\end{proof}

The proof of Theorem \ref{Main Thm - PSC with order} is carried out in a quite similar fashion, but the methods for detection are more sophisticated and rely on tools developed in \cite{crowley2025G2moduli}*{Section 7}.

\begin{proof}[Proof of Theorem \ref{Main Thm - PSC with order}]
    Let $M^4$ be a closed, oriented, smooth, psc manifold with signature $\mathrm{sign}(M) =: -\sigma < 0$.
    Choose $\sigma$-many disjointly embedded small, open discs $D^4 \hookrightarrow M$ and label them by $\alpha \in \{1,\dots,\sigma\}$.
    For each boundary component of the manifold $W := M \setminus \sqcup_{\alpha=1}^\sigma D^4_\alpha$, we have a choice to glue in a twisted bundle $\mathcal{DO}(1)$ to $W \times S^2$ or an untwisted bundle $D\mathcal{O}(1) \times S^2$.
    No matter the choice, we obtain a fibre bundle $E \rightarrow S^2$ with fibre $\CP^{2\, \sharp \sigma} \sharp M$, which has vanishing signature. 
    The structure group of this bundle is $\Diff(\CP^{2\, \sharp \sigma} \sharp M, W)$ by construction.

    For $\alpha \in \{1,\dots,\sigma\}$, let $E_\alpha \rightarrow S^2$ be the $\CP^{2\, \sharp \sigma} \sharp M$-fibre bundle that is obtained by gluing in the twisted bundle $\mathcal{DO}(1)$ to $W \times S^2$ along the boundary component corresponding to $\alpha$ and by gluing in the trivial bundle $D\mathcal{O}(1) \times S^2$ along the other components.

    A Mayer-Vietoris argument similar to the one that shows that homology groups of connected sums of closed oriented manifolds are additive in all degrees except zero and the top one shows that the images of the fundamental class of $\CP^1 \times S^2$ under the inclusions
    \begin{equation*}
        \iota_\alpha \colon \CP^1\times S^2 \hookrightarrow \mathcal{DO}(1) \hookrightarrow E_\alpha, \qquad \iota_\beta \colon \CP^1 \times S^2 \hookrightarrow D\mathcal{O}(1)\times S^2 \hookrightarrow E_\alpha
    \end{equation*}
    generate a subgroup isomorphic to $\Z^\sigma$ inside $H_4(E_\alpha;\Z)$.

    As in the proof of Theorem \ref{Main Thm - PSC}, we derive
    \begin{align*}
        p_1(T^{\vertical}E_\alpha) \cap \iota_{\alpha\,\ast}[\CP^1 \times S^2] &= \langle p_1(T^{\vertical}E_\alpha), \iota_{\alpha\,\ast}[\CP^1 \times S^2] \rangle = \langle -2xy, [\CP^1 \times S^2]\rangle = -2.   
    \end{align*}
    For the other embeddings, however, we derive
    \begin{align*}
        p_1(T^{\vertical}E_\alpha) \cap \iota_{\beta\,\ast}[\CP^1 \times S^2] &= \langle p_1(T^{\vertical}E_\alpha), \iota_{\beta\,\ast}[\CP^1 \times S^2] \rangle \\
        &= \langle p_1(T^\vertical E_\alpha |_{\CP^1 \times S^2}), \CP^1 \times S^2 \rangle \\
        &= \langle p_1(T^\vertical D\mathcal{O}(1) \times S^2 |_{\CP^1 \times S^2}), \CP^1 \times S^2 \rangle \\
        &= \langle \mathrm{pr}_1^\ast p_1(TD\mathcal{O}(1)|_{\CP^1}) , [\CP^1 \times S^2] \rangle = 0,
    \end{align*}
    because $p_1(T D\mathcal{O}(1)|_{\CP^1}) \in H^4(\CP^1,\Z) = 0$ vanishes for algebraic reasons.

    The (homotopy classes of) maps $f_\alpha \colon S^2 \rightarrow B\Diff( \CP^{2\, \sharp \sigma} \sharp M)$ representing the bundles $E_\alpha$ generate a subgroup $\mathcal{S}$ inside $\pi_2(B\Diff( \CP^{2\, \sharp \sigma} \sharp M))$. 
    Since $\CP^{2\, \sharp \sigma} \sharp M$ is simply connected, its third integral homology group vanishes by Poincar\'e duality.
    As the signature of $\CP^{2\, \sharp \sigma} \sharp M$ vanishes, and therefore $p_1(\CP^{2\, \sharp \sigma} \sharp M) = 0$ by Hirzebruch's signature theorem, all assumption of Proposition 7.3 in \cite{crowley2025G2moduli} with $m = 0$ in the notation of loc cit are satisfied and we deduce that the map that assigns to $[f_\alpha]$ the homomorphism $p_1(T^\vertical E_\alpha) \cap (\placeholder) \colon H_4(E_\alpha;\Z) \rightarrow \Z$ itself gives rise to a homomorphism 
    \begin{equation*}
        \Phi \colon \mathcal{S} \rightarrow \mathrm{Hom}\bigl(H_2( \CP^{2\, \sharp \sigma} \sharp M;\Z),\Z\bigr)
    \end{equation*}
    in the following sense: 
    It is proved in \cite{crowley2025G2moduli} that the homomorphism $p_1(T^\vertical E_\alpha) \cap (\placeholder) \colon H_4(E_\alpha;\Z) \rightarrow \Z$ factors through $H_2( \CP^{2\, \sharp \sigma} \sharp M;\Z)$, and we denote the latter homomorphism by $\bar{p}_1(T^\vertical E_\alpha) \cap (\placeholder)$. 
    Let $E^{\mathrm{univ}} \rightarrow B\Diff(\CP^{2 \,\sharp\, \sigma} \sharp M)$ be the universal $\CP^{2 \,\sharp\, \sigma} \sharp M$-fibre bundle and let $[f_\alpha],[f_\beta] \in \mathcal{S} \subseteq \pi_2(B\Diff(M))$. 
    For each representative of $[f_\alpha] + [f_\beta]$, denoted abusively by $(f_\alpha + f_\beta)$, we have that
    \begin{equation*}
        \bar{p}_1\bigl(T^\vertical(f_\alpha + f_\beta)^\ast E^{\mathrm{univ}} \bigr) \cap ( \placeholder ) = \bar{p}_1\bigl(T^\vertical f_\alpha^\ast E^{\mathrm{univ}} \bigr) \cap ( \placeholder ) + \bar{p}_1\bigl(T^\vertical f_\beta^\ast E^{\mathrm{univ}} \bigr) \cap ( \placeholder ).    
    \end{equation*}
    
    There is a subgroup\footnote{As the notation suggests, this subgroup is the image of the subgroup $\Z^\sigma \subseteq H_4(E_\alpha;\Z)$ described above under the homomorphism $H_4(E_\alpha;\Z) \rightarrow H_2( \CP^{2\, \sharp \sigma} \sharp M ;\Z)$ constructed in \cite{crowley2025G2moduli}.} $\Z^\sigma \subseteq H_2( \CP^{2\, \sharp \sigma} \sharp M ;\Z)$ generated by the $\CP^1$'s inside the added $D\mathcal{O}(1) = \CP^2 \setminus D^4$. We label these $\CP^1$'s with $\gamma\in\{1,\dots,\sigma\}$ according to the embedded disc $D^4_\gamma$ which they replace.
    Then the homomorphism $\Phi([f_\alpha])$ restricted to this subgroup is generated by
    \begin{equation*}
        \Phi(f_\alpha)([\CP^1_\gamma]) = p_1( T^\vertical E_\alpha ) \cap \iota_{\gamma\,\ast}[\CP^1 \times S^2] = \begin{cases}
            -2, & \text{if } \gamma = \alpha, \\
            \ \ 0, & \text{if } \gamma \neq \alpha.
        \end{cases}
    \end{equation*}
    This implies that $\mathcal{S} = \mathrm{span}\{[f_\alpha] \, : 1 \leq \alpha \leq \sigma\} \cong \Z^\sigma$ (via $\Phi$).

    The bundles $E_\alpha \rightarrow S^2$ can be equipped with a fibre metric with positive scalar curvature.
    Indeed, by the surgery theorem, see \cite{chernysh2004homotopy} or \cite{EbertFrenck2021Surgery}, we find a psc metric $g_0|_W$ on $W$ with product metric $\ground + \diff t^2$ near the boundary $\sqcup_{\gamma =1}^\sigma S^3_\gamma$.
    Let $g_{D\mathcal{O}(1)}^{\mathrm{ext}}$ denote the restriction of the fibre metric $g_{\mathcal{DO}(1)}^{\mathrm{ext}}$ to the fibre $D\mathcal{O}(1)$ over ${[1:0]} \in \CP^1 = S^2$.
    By construction, this is a psc metric that is of product form $\ground + \diff t^2$ near the boundary.
    
    We denote the constant extension of $g_{D\mathcal{O}(1)}^{\mathrm{ext}}$ to $D\mathcal{O}(1) \times S^2$ again by $g_{D\mathcal{O}(1)}^{\mathrm{ext}}$.
    Similarly, the constant extension of $g_0|_W$ to $W \times S^2$ is again denoted by $g_0|_W$.
    Then 
    \begin{equation*}
        g_{E_\alpha} = g_0|_W \cup g_{\mathcal{DO}(1)}^{\mathrm{ext}} \cup g_{D\mathcal{O}(1)}^{\mathrm{ext}} \cup \dots \cup g_{D\mathcal{O}(1)}^{\mathrm{ext}} 
    \end{equation*}
    is a fibre metric on $E_\alpha$ with positive scalar curvature.

    Thus, the maps $f_\alpha \colon S^2 \rightarrow B\Diff(\CP^{2\, \sharp \sigma} \sharp M)$ lift to the homotopy observer moduli space
    $\HObsModuli[\CP^{2\, \sharp \sigma} \sharp M]$ and hence generate a subgroup isomorphic to $\Z^\sigma$ there.
    Since the canonical comparison map
    \begin{equation*}
        \HObsModuli[\CP^{2\, \sharp \sigma} \sharp M] \rightarrow \ObsModuli[\CP^{2\, \sharp \sigma} \sharp M]
    \end{equation*}
    is a weak homotopy equivalence, the same is true for the observer moduli space.
\end{proof}

\begin{rem}
    Typical invariants to detect non-triviality of smooth fibre bundles are generalised \emph{Morita-Miller-Mumford} classes associated to characteristic classes.
    They arise by fibre-integrating the given characteristic class of the vertical-tangent bundle, see \cite{Ebert2014GeneralisedMMM} for further details.
    However, for $\CP^2_\Phi := \CP^2{\sharp}_\Phi S^4$, one can determine the cohomology ring explicitly and prove that the Morita-Miller-Mumford classes associated to Pontrjagin classes or Stiefel-Whitney classes agree with the ones of the product $\CP^2 \times \CP^1$.

    To give more details, the Morita-Miller-Mumford class associated to a characteristic class $\xi$ of the bundle $\CP^2 \xrightarrow{\iota} \CP^2_\Phi \xrightarrow{p} \CP^1$ is the unique cohomology class $\kappa_\xi := \kappa_\xi(\CP^2_\Phi)$ such that
    \begin{equation*}
        \kappa_\xi \cap [\CP^1] = H_\ast(p) \left( \xi(T^\vertical \CP^2_\Phi) \cap [\CP^2_\Phi] \right) \in H_\ast(\CP^1).
    \end{equation*}
    For algebraic reasons, $\kappa_\xi = 0$ unless the cohomological degree of $\xi$ is four, which implies that $\kappa_\xi \in H^0(\CP^1)$ in that case. 
    Since the fibrewise zero section
    \begin{equation*}
        \xymatrix{ \CP^1 \times \CP^1 \ar@{^{(}->}[rr]^{j} \ar[rd]_{\mathrm{pr}_2} && \mathcal{DO}(1) \subseteq \CP^2_\Phi \ar[ld]^p \\
        & \CP^1 & }
    \end{equation*}
    induces an isomormphism on the second homology group, we can find unique intergers $\lambda_1, \lambda_2 \in \Z$ such that
    \begin{equation*}
        \xi(T^\vertical \CP^2_\Phi) \cap [\CP^2_\Phi] = \lambda_1 \cdot H_2(j) ([\CP^1] \times 1) + \lambda_2 \cdot H_2(j)(1 \times [\CP^1]).
    \end{equation*}
    We deduce now from commutativity of the diagram that
    \begin{align*}
        H_2(p)\left( \xi(T^\vertical \CP^2_\Phi) \cap [\CP^2_\Phi] \right) = H_2(\mathrm{pr}_2)\left(\lambda_1 ([\CP^1] \times 1) + (1 \times [\CP^1])\right) = \lambda_2 [\CP^1].
    \end{align*}
    Thus, $\kappa_\xi = \lambda_2 \in H^0(\CP^1)$.

    Recall that $c \in H^2(\CP^1)$ denotes the restriction of the universal first Chern class and recall that $c \cap [\CP^1] = 1$.
    From the naturality of the cap product together with the observation that $H_0(p) = \id$\footnote{if domain and target are identified with $H_0(\mathrm{pt})$ via $H_0(\CP^2_\Phi) \xleftarrow[\cong]{H_0(\mathrm{incl})} H_0(\mathrm{pt}) \xrightarrow[\cong]{H_0(\mathrm{incl})} H_0(\CP^1)$. }, we derive
    $\lambda_2 = H^2(p)(c) \cap \left( \xi(T^\vertical \CP^2_\Phi) \cap [\CP^2_\Phi]  \right)$.
    The Poincare dual of $H^2(p)(c)$ is the homology class $H_4(\iota)([\CP^2])$, where $\iota$ denotes a fibre inclusion, so
    \begin{align*}
        \kappa_\xi &= \xi(T^\vertical \CP^2_\Phi) \cap H^2(p)(c) \cap [\CP^2_\Phi] =  \xi(T^\vertical \CP^2_\Phi) \cap H_4(\iota)([\CP^2]) \\
        &= H_0(\iota)\left( H^4(\iota) (\xi(T^\vertical \CP^2_\Phi)) \cap [\CP^2] \right) = H_0(\iota) \left( \xi(T \CP^2) \cap [\CP^2] \right)
    \end{align*}
    only depends on the fibre. 
\end{rem}


\section{Proof of Theorem \ref{Main Thm - PSec}}\label{Section - Positive Sectional Curvature}

To equip $\CP^2_\Phi$ from the previous section with a fibrewise positive sectional curvature metric it will be more convenient to use a different description for $\CP^2_\Phi$ so that we do not need to rely on gluing methods.

\begin{definition}\label{Def - TwistedCPnBundles}
    For $n \geq 2$, we define  
    \begin{equation*}
        \CP^n_\Phi := S^3  \times_{S^1, (\cdot)^{-1}} \CP^n = \bigl(S^3 \times \CP^n \bigr)/\sim 
    \end{equation*}
    where $\bigl(z,[w_0:w_1:w_2:\dots : w_n]\bigr) \sim ( \mathrm{e}^{\I \theta} \cdot z, [\mathrm{e}^{\I \theta} w_0: \mathrm{e}^{\I \theta} w_1: w_2 : \dots : w_n]). $
\end{definition}

Completely analogous to Lemma \ref{Lemma - TransTwistedDiscBundle} one proves the next result. 
We will use the notation from Definition \ref{Def - Sections of Hopf fibration}, and remind the reader that $\sigma_j$ are local sections of the Hopf fibration, so that $\sigma_j([z])$, whenever it is defined, and $z$ differ by a unique element in $S^1$.  
\begin{lemma}\label{Lemma - TransFctTwistCPn}
    The map $p \colon \CP^n_\Phi \rightarrow \CP^1$ induced by the projection to first component followed by the Hopf fibration $S^3 \rightarrow \CP^1$ endows $\CP^n_\Phi$ with a fibre bundle structure.
    Over $U_j = \{[a_0:a_1]\,:\, a_j \neq 0\}$, the following maps
    \begin{align*}
        \Psi_j \colon p^{-1}(U_j) &\rightarrow U_j \times \CP^n, \\
        \left\lbrack z, [w_0:\dots:w_n] \right\rbrack &\mapsto \left( [z], [ \sigma_j([z])z^{-1} w_0 : \sigma_j([z])z^{-1} w_1: w_2: \dots : w_n] \right)
    \end{align*}
    are local trivialisations.
    The corresponding transition functions are given by
    \begin{equation*}
        g_{01}([a_0:a_1])([w_0:\dots :w_n]) = \bigl( \sigma_0\sigma_1^{-1}([a_0:a_1])w_0 : \sigma_0\sigma_1^{-1}([a_0:a_1])w_1 : w_2: \dots :w_n \bigr).
    \end{equation*}
    The structure group of $\CP^2_\Phi \rightarrow \CP^1$ can be reduced to $\Diff(\CP^n,D^{2n})$, the group of all diffeomorphisms that are the identity on $D^{2n} = \{[w_0: \dots : w_{n-1}:1] \,: \, |w_0|^2 + \dots + |w_{n-1}|^2 \leq 1\}$.
\end{lemma}

Since the chosen action in Definition \ref{Def - TwistedCPnBundles} is isometric with respect to the Fubini-Study metric $\gFubiniStudy$, the bundle $\CP^n_\Phi$ carries a fibre metric with positive sectional curvature.
Together with Lemma \ref{Lemma - TransFctTwistCPn} we get a classifying map
\begin{equation*}
    f_{\CP^n_\Phi,\gFubiniStudy} \colon S^2 = \CP^1 \rightarrow \HObsModuli[\CP^n][\sect > 0].
\end{equation*}
The first part of Theorem \ref{Main Thm - PSec} now specifies to the following result.
\begin{theorem}
    Let $C$ be a diffeomorphism invariant curvature condition that is implied by positive sectional curvature.
    The element $[f_{\CP^n_\Phi,\gFubiniStudy}] \in \pi_2\left(\HObsModuli[\CP^n][C],[\gFubiniStudy]\right)$ has at least order $n+1$ if $n$ is even and at least order $(n+1)/2$ if $n$ is odd.
\end{theorem}
\begin{proof}
    Under the forgetful map $\HObsModuli[\CP^n][\sect > 0] \rightarrow B\Diff(\CP^n)$ this element gets mapped to $[f_{\CP^n_\Phi}]$, the class that is represented by the classifying map of the bundle $
    \CP^n_\Phi$.
    Since the transition functions of $\CP^n_\Phi$ take values in $\PUn[n+1]$, this class lifts to an element $\pi_2(B\PUn[n+1])$.
    As $\mathrm{det}(g_{01}) \colon U_0 \cap U_1 \rightarrow \Un[1]$ is a map of degree $\pm 2$ it follows that $[f_{\CP^n_\Phi}]$ corresponds to image of the element $\pm 2 \in \pi_2(B\PUn[n+1]) \cong \Z_{n+1}$ under the homomorphism induced by $B\PUn[n+1] \rightarrow B\Diff(\CP^n)$.
    
    It was shown in \cite{Sasao1974hAutCPn} that the canonical map $B\PUn[n+1] \rightarrow B\hAut(\CP^n)$, which of course factors through $B\Diff(\CP^n)$, induces an isomorphism between their second homotopy groups.
    Thus $\pi_2(B\PUn[n+1]) \rightarrow \pi_2(B\Diff(\CP^n))$ is injective, which establishes the claim for positive sectional curvature.

    For any other diffeomorphism invariant curvature condition $C$ that is implied by positive sectional curvature, the forgetful map $\HObsModuli[\CP^n][\sect > 0] \rightarrow B\Diff(\CP^n)$ factors through the forgetful map $\HObsModuli[\CP^n][\mathrm{sec}>0] \rightarrow \HObsModuli[\CP^n][C]$, so the order of the element $[f_{\CP^n_\Phi}] \in \pi_2(\HObsModuli[\CP^n][C],[g_{FS}])$ is bounded from below by the same constant.
\end{proof}
Exploiting information about the rational homotopy groups of $B\hAut(\CP^n)$, we can deduce the existence of non-trivial elements in higher homotopy groups of the observer moduli space in a stable range. 
This will prove the second part of Theorem \ref{Main Thm - PSec}.

\begin{proof}[Proof of Theorem \ref{Main Thm - PSec} (2)]
   Since the unitary group $\Un[n+1]$ acts isometrically on $(\CP^n,\gFubiniStudy)$, 
   the bundle $f^\ast E\Un[n+1] \times_{\Un[n+1]} \CP^n$ carries a fibrewise positive sectional curvature metric for every element $[f] \in \pi_{2j}(B\Un[n+1])$.
    The choice of such a fibre metric yields an element
    \begin{equation*}
        [f_{\gFubiniStudy}] \in \pi_{2j}\left( \HModuli[\CP^n][\sect > 0],[\gFubiniStudy] \right)
    \end{equation*}
    that lifts the image of $[f]$ under the homomorphism induced by the canonical map $B\Un[n+1] \rightarrow B\Diff(\CP^n)$.

   We first prove that $[f_{\gFubiniStudy}]$ can be chosen to have infinite order: 
   The classifying map of the canonical inclusion $B\PUn[n+1] \rightarrow B\hAut(\CP^n)$, which factors of course through $B\Diff(\CP^n)$, induces an isomorphism on all rational homotopy groups, see \cite{Sasao1974hAutCPn}.
    The rational homotopy groups of $B\hAut(\CP^n)$ are given by
    \begin{equation*}
        \pi_{2j}\left( B\hAut(\CP^n) \right) \otimes \Q \cong \Q \qquad \text{ for all  } 2 \leq j \leq n+1 
    \end{equation*}
    and all other are zero, see \cite{Sullivan1977Infinitesimal}*{p.314}.
    Thus, if $[f] \in \pi_{2j}(B\Un[n+1]) \cong \pi_{2j}(B\PUn[n+1])$ has infinite order, then the lift $[f_{\gFubiniStudy}]$ has infinite order, too.

    It remains to show that a finite multiple of $[f_{\gFubiniStudy}]$ lifts to the (homotopy) observer moduli space.
    From the long exact sequence on (rational) homotopy groups induced by the fibre bundle $\SOr[2n] \rightarrow \mathrm{Fr}^+(\CP^n) \rightarrow \CP^n$ we deduce for all $4 \leq k \leq 2n-1$ that
    \begin{equation*}
        \pi_{k-1}(\mathrm{Fr}^+(\CP^n))\otimes {\Q} \cong \pi_{k-1}(\SOr[2n])\otimes \Q \cong \begin{cases}
           \Q, & \text{if } k \equiv 0 \mod 4,\\
           0, & \text{else}.
        \end{cases}
    \end{equation*} 
   The second part of Proposition \ref{Prop - HModuliProperties} now implies that a finite multiple of $[f_{\gFubiniStudy}]$ lifts to $\pi_{2j}(\HObsModuli[\CP^n][\sect > 0])$ if $2\leq j\leq n-1$ is odd. 
   This lift, of course, automatically has infinite order.

   For any diffeomorphism invariant curvature condition $C$ that is implied by positive sectional curvature, the forgetful map $\HObsModuli[\CP^n][\sect > 0] \rightarrow B\Diff(\CP^n)$ factors through the forgetful map $\HObsModuli[\CP^n][\sect > 0] \rightarrow \HObsModuli[\CP^n][C]$, so the elements $[f_{\CP^n_\Phi}] \in \pi_{2j}(\HObsModuli[\CP^n][C],[g_{FS}])$ is also of infinite order as long as $2\leq j\leq n-1$ is odd. 
\end{proof}

\begin{rem}
    A careful analysis of the involved homomorphisms even shows that the composition 
    \begin{equation*}
        \pi_{2j}(B\PUn[n+1]) \otimes \Q \rightarrow \pi_{2j}(B\Diff(\CP^n)) \otimes \Q \rightarrow \pi_{2j-1}(\mathrm{Fr}^+(\CP^n))\otimes \Q
    \end{equation*}
    is injective if $2\leq j \leq n+1$ is even.
    Hence, this approach cannot produce more non-trivial elements in the rational homotopy groups of the observer moduli space of positive sectional curvature metrics on $\CP^n$.
\end{rem}

\bibliographystyle{alpha} 
\bibliography{Literatur}

\vspace{1cm}

\end{document}